\documentclass[reqno,11pt,english]{amsart}
\usepackage[T1]{fontenc}
\usepackage[latin9]{inputenc}
\usepackage{color}
\usepackage{amstext}
\usepackage{amsthm}
\usepackage{amssymb}

\makeatletter
\usepackage{amsmath,amsfonts,amssymb,amsthm,epsfig}

\usepackage{color}
\usepackage[final,colorlinks,linkcolor=blue,anchorcolor=red,citecolor=blue]{hyperref}
\def\R{\mathbb R}
\def\N{\mathbb N}

\def\al{\alpha}
\def\be{\beta}
\def\ga{\gamma}
\def\de{\delta}
\def\ep{\epsilon}
\def\la{\lambda}
\def\La{\Lambda}

\def\ta{\theta}
\def\var{\varphi}
\def\om{\omega}
\def\na{\nabla}
\def\Ga{\Gamma}  % Euler functions
\def\Om{\Omega}  % domains
\def\De{\Delta}      % Laplacian operators
\def\wq{\infty}
\def\pa{\partial}
\def\divergence{\text{\rm div}\,}
\def\Id{{\rm Id}\,}
\def\loc{\text{\rm loc}}

\newcommand{\D}{{\rm d}}%  integral sign d

\newcommand{\medint}{-\kern -,375cm\int}         %  average integral
\newcommand{\medintinrigo}{-\kern -,315cm\int}
\newcommand{\wto}{\rightharpoonup}                %  weak convergence

 \def\Sd{\text{\rm div}\,} %%%Sd=divergence=san du

\numberwithin{equation}{section}
\newtheorem{theorem}{Theorem}[section]
\newtheorem*{theorem*}{Theorem}  %Theorems without number

\newtheorem*{conclusion*}{Conclusin}

\newtheorem*{conjecture*}{Conjecture}
\newtheorem{corollary}[theorem]{Corollary}
\newtheorem*{corollary*}{Corollary}

\newtheorem{lemma}[theorem]{Lemma}
\newtheorem*{lemma*}{Lemma}

\newtheorem*{notation*}{Notation}

\newtheorem{proposition}[theorem]{Proposition}
\newtheorem*{proposition*}{Proposition}
\newtheorem{remark}[theorem]{Remark}
\newtheorem*{remark*}{Remark}

\newtheorem{example}[theorem]{Example}
\newtheorem*{example*}{Example}                %%%%% example can not use, due to conflict with Lyx.
\theoremstyle{definition}
\makeatother

\usepackage{babel}
\begin{document}
\title[]{The Lamm-Rivi\`ere system I: $L^p$ regularity theory}

 \author[C.-Y. Guo, C.-L. Xiang and G.-F. Zheng]{Chang-Yu Guo, Chang-Lin Xiang$^\ast$ and Gao-Feng Zheng}

\address[Chang-Yu Guo]{Research Center for Mathematics and Interdisciplinary Sciences, Shandong University 266237,  Qingdao, P. R. China and Institute of Mathematics, \'Ecole Polytechnique F\'ed\'erale de Lausanne (EPFL), Station 8,  CH-1015 Lausanne, Switzerland}
\email{changyu.guo@email.sdu.edu.cn}

\address[Chang-Lin Xiang]{School of Information and Mathematics, Yangtze University, Jingzhou 434023, P. R. China}
\email{changlin.xiang@yangtzeu.edu.cn}

\address[Gao-Feng Zheng]{School of Mathematics and Statistics, Central China Normal University, Wuhan 430079,  P. R.  China}
\email{gfzheng@mail.ccnu.edu.cn}

\thanks{*Corresponding author: Chang-Lin Xiang}
\thanks{C.-Y. Guo was supported by Swiss National Science Foundation Grant 175985 and the Qilu funding of Shandong University (No. 62550089963197). The corresponding author C.-L. Xiang is financially supported by the National Natural Science Foundation of China (No. 11701045) and  the Yangtze Youth Fund (No. 2016cqn56). G.-F. Zheng is supported by the National Natural Science Foundation of China (No. 11571131).}

\begin{abstract}
Motived by the heat flow and bubble analysis of biharmonic mappings, we study further regularity issues of the fourth order Lamm-Rivi\`ere system
$$\De^{2}u=\De(V\cdot\na u)+{\rm div}(w\na u)+(\na\om+F)\cdot\na u+f$$
in  dimension four, with an inhomogeneous term $f$ which belongs to some natural function space. We obtain optimal higher order  regularity and sharp H\"older continuity of weak solutions. Among several applications, we derive weak compactness for  sequences of weak solutions with uniformly bounded energy, which generalizes the weak convergence theory of approximate biharmonic mappings.
\end{abstract}

\maketitle

{\small
\smallskip
\newline
\tableofcontents}
\bigskip

\section{Introduction}

\subsection{Background and motivation}

In the calculus of variations, finding regular critical points of the variational functional
$$u\mapsto \int F(x,u(x), Du(x))\D x$$
with quadratic growth has been one of the most attractive topics. Among many other interesting geometric models, those related to conformally invariant variational problems, are of particular interests.

A fundamental work of Morrey \cite{Morrey-1948} shows that minimizers in $W^{1,2}(B^2, N)$, $N\subset \R^m$, of the standard Dirichlet energy are locally H\"older continuous and thus are as smooth as the (embedded) Riemannian manifold $N$. However, when the domain has higher dimensions (greater than or equal to three), one loses conformal/scaling invariance of minimizers and in this case, there exist discontinuous minimizers, much less to say about general critical points. A well-known conjecture along this direction was formulated by Hildebrand \cite{Hildbrandt-1980}:
\begin{conjecture*}\label{conjecture:Hildbrandt}
Critical points of coercive conformally invariant Lagrangian with quadratic growth are regular.
\end{conjecture*}
%A well-known conjecture along this direction was formulated by Hildbrandt \cite{Hildbrandt-1980} in the late 1970s:\emph{ critical points of coercive conformally invariant Lagrangian with quadratic growth are regular.}

In his pioneer work \cite{Helein-2002}, Helein confirmed  this conjecture  in the case of weakly  harmonic mappings: every weakly harmonic mappings from the two dimensional disk $B^2\subset \R^2$ into any closed manifold are smooth via the nowadays well-known moving frame method.

In 2007, this conjecture was fully settled down  by Rivi\`ere in his remarkable work \cite{Riviere-2007}. More precisely, in  \cite{Riviere-2007},  he proposed the general second order linear elliptic  system
\begin{equation}\label{eq:Riviere system}
	-\Delta u=\Omega\cdot \nabla u \qquad \text{in }B^2
\end{equation}
where $u\in W^{1,2}(B^2, \R^m)$ and $\Omega=(\Omega_{ij})\in L^2(B^2,so_m\otimes \Lambda^1\R^2)$. As was verified in \cite{Riviere-2007}, \eqref{eq:Riviere system} includes
the Euler-Lagrange equations of  critical points of all second order conformally invariant variational functionals which  act on mappings $u\in W^{1,2}(B^2,N)$ from  $B^2\subset \R^2$ into a closed Riemannian manifold $N\subset \R^m$. The approach of Rivi\`ere involves finding a map $A\in L^{\wq}\cap W^{1,2}(B^2, Gl(m))$ and $B\in W^{1,2}(B^2, M_m)$
satisfying $\nabla A-A\Omega=\na^{\bot}B$,
 such that system \eqref{eq:Riviere system} can be written  equivalently as the conservation law
\begin{equation}\label{eq:conservation law of Riviere}
	\divergence(A\nabla u+B\nabla^{\bot} u )=0.
\end{equation}
Then the continuity of weak solutions of system \eqref{eq:Riviere system} follows rather easily from the conservation law \eqref{eq:conservation law of Riviere}. As \eqref{eq:Riviere system} includes the equations of weakly harmonic mappings from  $B^2$ into $N$, this recovered the regularity result of  H\'elein  \cite{Helein-2002}. In fact, Rivi\`ere's work has far more applications beyond conformally invariant problems; see \cite{Riviere-2011,Riviere-2012} for a comprehensive overview.

Starting from the celebrated work of Eells and Sampson~\cite{Eells-Sampson-1964}, there have been great attempts to find regular solutions for the heat flow of harmonic mappings (see for instance \cite{Struwe-1985,Struwe-1988,Liu-2003-ARMA,Moser-2015-TAMS} and the references therein). This leads to the general consideration of the inhomogeneous Rivi\`ere system
\begin{equation}\label{eq:inhomogenuous system ST}
-\Delta u=\Omega\cdot \nabla u+f\qquad \text{in }B^2,
\end{equation}
where the drift term $f\colon B^2\to \R^m$ belongs to certain natural function space. In case of heat flow of harmonic mappings, $f\in L^p(B^2,\R^m)$ shall denote the first order partial derivative of the flow with respect to time. Two basic topics, with large mathematical interest, related to \eqref{eq:inhomogenuous system ST} are the weak compactness of Palais-Smale sequences and the energy identity (or bubble analysis). The energy identity quantifies the limiting behaviour of certain energy of a sequence of weak solutions. To be more precise, let $\{u_n\}_{n\in \N}$ be a sequence of weak solutions (say, to the harmonic mapping system) with uniformly bounded energy, for which we denoted by $E(u_n)$. In general, a subsequence of $\{u_n\}$ shall converge weakly to some limiting map $u$, but the convergence does not necessarily have to be strong. One can show, with some effort, that away from a finite set $\Sigma=\{x_1,\cdots,x_k\}$, the converge $u_n\to u$ will be strong. Moreover, the loss of energy during the limiting process happens exactly because of energy concentration at these finite points $x_i$, $i=1,\cdots,k$.

The study of energy identity for harmonic mappings was initiated by Sacks and Uhlenbeck in the seminal work \cite{Sacks-Uhlenbeck-1981} and then attracted great attention in geometric analysis of various mappings and equations. Based on the new method of Rivi\`ere \cite{Riviere-2007}, energy identity for the general system \eqref{eq:inhomogenuous system ST} was obtained very recently in \cite{Laurain-Riviere-2014-APDE,Lamm-Sharp-2016-CPDE}. An important intermediate step towards  these results is to establish a higher order $L^p$-regularity theory for weak solutions of the system \eqref{eq:inhomogenuous system ST}, which was done in the very interesting work of Sharp and Topping \cite{Sharp-Topping-2013-TAMS}; see also \cite{Moser-2015-TAMS} for applications of higher $L^p$-regularity theory in the study of heat flow of harmonic mappings.

%\textbf{Re-organize this paragraph}: In \cite{Rupflin-2008}, Rupflin proved that for $p\in (1,2)$, weak solutions of \eqref{eq:inhomogenuous system ST} are locally H\"older continuous with an optimal exponent $\alpha=2(1-\frac{1}{p})$.  Then in an interesting work of Sharp and Topping \cite{Sharp-Topping-2013-TAMS}, the authors investigated in depth what sort of regularity and compactness properties one can deduce for solutions of the system \eqref{eq:inhomogenuous system ST}. They proved that when $p\in (1,2)$, each weak solution $u\in W^{1,2}(B_1,\R^m)$ of \eqref{eq:inhomogenuous system ST} actually belongs to $W^{2,p}_{\loc}(B_1)$, and hence recovered the H\"older continuity result of Rupflin \cite{Rupflin-2008}. Furthermore, under certain smallness assumption on the $L^2$-norm of $\Omega$, they derived optimal global estimates for the weak solutions of \eqref{eq:inhomogenuous system ST}. As an application, they obtained compactness result for the inhomogeneous system \eqref{eq:inhomogenuous system ST}.

Moving to four dimensions and taking into account of the conformal invariance in $\R^4$, in order to obtain smooth mappings $u\colon B^4\to N\hookrightarrow \R^m$, it is natural to consider critical points of the bi-energy (or intrinsic bi-energy), that is, critical points of the $L^2$-norm of $\Delta u$ (or $(\Delta u)^T$\footnote{Here $(\Delta u)^T$ denotes the projection of $\Delta u$ into $T_uN$.}, respectively). These critical points are called extrinsic (intrinsic, respectively) biharmonic mappings and they form a natural generalization of harmonic mappings.  Chang, Wang and Yang \cite{Chang-W-Y-1999} initiated the study of  regularity theory of extrinsic biharmonic mappings from the $n$-dimensional Euclidean ball $B^n$ into Euclidean spheres and proved smoothness of these mappings.  Shortly after that, Wang  developed a regularity theory of both extrinsic and intrinsic biharmonic mappings into general closed Riemannian manifolds in a series of pioneer works \cite{Wang-2004-CV,Wang-2004-MZ,Wang-2004-CPAM} via the method of Coulomb frames.

Similar to harmonic mappings, there has been great interest to find regular solutions for the heat flow of biharmonic mappings (see for instance \cite{Lamm-2004,Lamm-2005,Wang-2012-JGA,Hineman-Huang-Wang-2014-CVPDE}). This leads to the general consideration of the inhomogeneous Lamm-Rivi\`ere system
\begin{eqnarray}\label{eq:inhomogenous Lamm-Riviere system 1}
\De^{2}u=\De(V\cdot\na u)+{\rm div}(w\na u)+(\na\om+F)\cdot\na u+f &  & \text{in }B^{4},
\end{eqnarray}
where all the involved coefficients belong to some natural function spaces. When $f=0$, the homogeneous  system \eqref{eq:inhomogenous Lamm-Riviere system 1}  was first introduced by Lamm and Rivi\`ere \cite{Lamm-Riviere-2008}. It includes both   Euler-Lagrange equations of the extrinsic and intrinsic biharmonic mappings from Euclidean balls into Riemannian manifolds as well as their variants such as approximate biharmonic mappings. One particular motivation for Lamm and Rivi\`ere to consider system \eqref{eq:inhomogenous Lamm-Riviere system 1} is to extend the new powerful method of Rivi\`ere \cite{Riviere-2007} to fourth order system and to give a unified treatment of the regularity theory for the above mentioned mapping classes. Based on the fundamental work of Lamm and Rivi\`ere \cite{Lamm-Riviere-2008}, the first two authors of the present paper established the weak compactness of Palais-Smale sequences in \cite{Guo-Xiang-2019-Boundary} and it remains a natural problem to study the energy identity (or bubble analysis) for sequences of weak solutions of \eqref{eq:inhomogenous Lamm-Riviere system 1}, extending the corresponding results for (approximate) biharmonic mappings obtained in \cite{Hornung-Moser-2012,Wang-Zheng-2012-JFA,Laurain-Riviere-2013-ACV}.

In the present paper, we aim at establishing a higher order $L^p$-regularity theory for weak solutions of the inhomogeneous Lamm-Rivi\`ere system \eqref{eq:inhomogenous Lamm-Riviere system 1}, akin to that of Sharp and Topping \cite{Sharp-Topping-2013-TAMS} for the inhomogeneous Rivi\`ere system \eqref{eq:inhomogenuous system ST}. In a following-up work, we shall apply the $L^p$-regularity theorems to establish the energy identity for system \eqref{eq:inhomogenous Lamm-Riviere system 1}.

\subsection{Main results}

Let $B_{r}\subset\R^{4}$ be an open ball with radius $r$, $m\in \N$ and  $u\in W^{2,2}(B_{10},\R^{m})$  a weak solution of the  inhomogeneous Lamm-Rivi\`ere system
\begin{eqnarray}\label{eq:inhomogenous Lamm-Riviere system}
\De^{2}u=\De(V\cdot\na u)+{\rm div}(w\na u)+(\na\om+F)\cdot\na u+f &  & \text{in }B_{10},
\end{eqnarray}
where
\begin{equation}\label{eq: regularity of coefficients}
\begin{aligned}
 &   V\in W^{1,2}(B_{10},M_{m}\otimes\Lambda^{1}\R^{4}),\quad\quad w\in L^{2}(B_{10},M_{m})\\
 &  \om\in L^{2}(B_{10},so_{m}),\quad\quad   F\in L^{\frac{4}{3},1}(B_{10},M_{m}\otimes\Lambda^{1}\R^{4})
\end{aligned}
\end{equation}
  and $f\in L\log L(B_{10},\R^{m})$.  For the definitions of the Lorentz function spaces $L^{\frac{4}{3},1}$ and $L\log L$, see section \ref{sec: preliminaries}.

%Inspired by the two dimensional case,
%there are also further development on the regularity theory of the fourth order elliptic system \eqref{eq:inhomogenous Lamm-Riviere system}. As mentioned earlier,
%Lamm and Rivi\`ere \cite{Lamm-Riviere-2008} successfully extended the conservation law of \eqref{eq:Riviere system} to the fourth order system \eqref{eq:inhomogenous Lamm-Riviere system}.
We begin our discussion by recording the following fundamental result of Lamm and Rivi\`ere \cite{Lamm-Riviere-2008}.
\begin{theorem*}[Lamm and Rivi\`ere, \cite{Lamm-Riviere-2008} ]\label{prop: Lamm-Riviere}
  For any $m\in \N$, there exist  constants $C_m>0$ and $\ep_m>0$  such that  If
  \begin{equation}\label{eq:smallness assumption}
\|V\|_{W^{1,2}(B_{{10}})}+\|w\|_{L^{2}(B_{{10}})}+\|\om\|_{L^{2}(B_{10})}+\|F\|_{L^{4/3,1}(B_{10})}<\ep_{m},
\end{equation} then there exist  $A\in W^{2,2}\cap L^{\wq}(B_{8},M(m))$ and $B\in W^{1,4/3}(B_{8},M(m)\otimes\wedge^{2}\R^{4})$
satisfying
\begin{eqnarray*}\label{eq:condition for A and B}
\na\De A+\De AV-\na Aw+AW={\rm curl}(B) &\text{ in } B_8,
\end{eqnarray*}
where ${\rm curl}(B)=\sum_{l}\partial_{x_l}B_{lk}\partial_{x_k}$. Moreover,
\begin{equation}\label{eq: A-B small}
\begin{aligned}
&\|A\|_{W^{2,2}(B_{8})}+\|{\rm dist}(A,SO_{m})\|_{L^{\wq}(B_{8})}+\|B\|_{W^{1,4/3}(B_{8})}\\
&\le C_{m}\left(\|V\|_{W^{1,2}(B_{{10}})}+\|w\|_{L^{2}(B_{{10}})}+\|\om\|_{L^{2}(B_{{10}})}
+\|F\|_{L^{4/3,1}(B_{10})}\right).
\end{aligned}
\end{equation}
Consequently,  $u$ solves \eqref{eq:inhomogenous Lamm-Riviere system}  if and only if it satisfies the conservation law
\begin{equation}\label{eq:conservation law of Lamm Riviere}
\De(A\De u)=\divergence(K)+Af \qquad \text{ in } B_8,
\end{equation}
where
\begin{equation}\label{eq: K}
K=2\na A\De u-\De A\na u+Aw\na u-\na AV\cdot\na u+A\na(V\cdot\na u)+B\cdot\na u.
\end{equation}
\end{theorem*}

Combining \eqref{eq: regularity of coefficients} with the regularity of $A$ and $B$, one easily verifies that $K\in L^{\frac{4}{3}, 1}(B_8)$. Thus $A\De u\in W^{1,\frac{4}{3},1}(B_8)$, which in turn gives $\De u\in W^{1,\frac{4}{3},1}(B_8)$. By elliptic regularity theory,  this implies $u\in W^{3,\frac{4}{3},1}(B_8)$ and thus $u\in C(B_8)$ by Lorentz-Sobolev embedding theorems. For details, see the proof of Theorem 1 of  \cite{Lamm-Riviere-2008} or Theorem \ref{thm:optimal local estimate LlogL} below.
 %Note however that $W^{3,\frac{4}{3},1}$ does not embed into any H\"older spaces.

In a very recent work \cite{Guo-Xiang-2019-Boundary}, the first two authors of the present paper further established  H\"older continuity of weak solutions of \eqref{eq:inhomogenous Lamm-Riviere system} (with $f=0$) by deriving a decay estimate via  the conservation law \eqref{eq:conservation law of Lamm Riviere}. We mention that it is also possible to obtain the H\"older continuity without using convervation law, for details see  \cite{Guo-Xiang-2019-Higher}.

Our first theorem deals with the optimal H\"older continuity of weak solutions to  \eqref{eq:inhomogenous Lamm-Riviere system}.

\begin{theorem}[H\"older continuity]\label{thm:optimal Holder exponent for inho Lamm-Riviere}
	Let $u\in W^{2,2}(B_{10},\R^m)$ be a weak solution of \eqref{eq:inhomogenous Lamm-Riviere system} and assume $f\in L^p(B_{10})$ for $p\in (1,\frac{4}{3})$. Then $u$ is locally $\al$-H\"older continuous with exponent $\al=4(1-\frac{1}{p})$.

Moreover, there exists  $C=C(p,m)>0$ such that for all $0<r<1$, there holds
 \begin{equation}\label{eq: decay estimate for 4th order}
    \begin{aligned}
    \|\na u\|_{L^{4,2}(B_r)}+\|\De u\|_{L^2(B_r)}\leq  Cr^{\al}\left( \|\na u\|_{L^{4,2}(B_1)}+\|\De u\|_{L^2(B_1)}+\|f\|_{L^p(B_1)}\right).
      \end{aligned}
     \end{equation}
\end{theorem}

The H\"older continuity is optimal, as one can see from the simplest case $\De^2 u=f$. Moreover, as one can easily notice, applying Theorem \ref{thm:optimal Holder exponent for inho Lamm-Riviere}  to the case $f\equiv 0$ yields that every weak solution $u\in W^{2,2}(B_{10})$ of the Lamm-Rivi\`ere system\eqref{eq:inhomogenous Lamm-Riviere system} is locally $\al$-H\"older continuous for all $\al\in (0,1)$. On the other hand, the H\"older continuity is the best possible regularity that one can expect for weak solutions of \eqref{eq:inhomogenous Lamm-Riviere system} (even when $f\equiv 0$). Indeed, we shall construct a weak solution $u\colon B\to \R$, $B\subset \R^4$, which belongs to $C^{0,\alpha}(B)\cap W^{2,2}(B)$ for any $\alpha\in (0,1)$, of the system
\[\Delta^2 u=\Delta(V\cdot \nabla u)\]
for some $V\in W^{1,2}(B,\R^4)$. But $u$ fails to be (locally) Lipschitz continuous in $B$. For details, see Remark \ref{exam:non-Lipschitz weak solutions} below.

%We next turn to the higher regularity of weak solutions of \eqref{eq:inhomogenous Lamm-Riviere system}. Intuitively, for any $f\in L^p$, $p>1$, the fourth order system \eqref{eq:inhomogenous Lamm-Riviere system} seems to imply a fourth order regularity, i.e.,  $u\in W^{4,p}_{\loc}$. However, it turns out to be impossible in general.
In our second theorem, we derive optimal higher order regularity of weak solutions.

\begin{theorem}[local $L^p$ estimates]\label{thm:optimal global estimate for inho Lamm-Riviere}
Let $u\in W^{2,2}(B_{10},\R^m)$ be a weak solution of \eqref{eq:inhomogenous Lamm-Riviere system} with $f\in L^p(B_{10})$ for $p\in (1,\frac{4}{3})$. Then $$u\in W_{\loc}^{3,\frac{4p}{4-p}}(B_{10}).$$ Moreover,
 there exist $\ep=\ep(p,m)>0$ and $C=C(p,m)>0$ such that if the smallness condition \eqref{eq:smallness assumption} is satisfied with $\ep_m=\ep$, then
    \begin{equation}\label{eq:optimal third order}
    \|u\|_{W^{3,\frac{4p}{4-p}}(B_{\frac{1}{2}})}\le C\left(\|f\|_{L^{p}(B_{1})}+\|u\|_{L^{1}(B_{1})}\right).
    \end{equation}

If, in addition, we assume $V\in W^{2, \frac{4}{3}}(B_{10})$ and $ w\in W^{1,\frac{4}{3}}(B_{10})$, then $$u\in W_{\loc}^{4,p}(B_{10}),$$ and
    \begin{equation}\label{eq:optimal fourth order}
    \|u\|_{W^{4,p}(B_{\frac{1}{2}})}\le C\left(\|f\|_{L^{p}(B_{1})}+\|u\|_{L^{1}(B_{1})}\right).
    \end{equation}
    %\end{enumerate}
\end{theorem}

Theorem \ref{thm:optimal global estimate for inho Lamm-Riviere} can be regarded as a counterpart of Sharp-Topping \cite[Theorem 1.1]{Sharp-Topping-2013-TAMS} to the fourth order system \eqref{eq:inhomogenous Lamm-Riviere system}. Some special cases of Theorem \ref{thm:optimal global estimate for inho Lamm-Riviere} can be found in the literature. In order to study the global existence of extrinsic biharmonic map flow,  Lamm and Rivi\`ere \cite[Lemma 3.1]{Lamm-Riviere-2008} proved a regularity  result for $f\in L^{2}$ under some special conditions on $V,w,\om, F$. Wang and Zheng  \cite[Lemma 2.3 ]{Wang-Zheng-2012-JFA} proved $W^{4,p}$ regularity for approximate extrinsic biharmonic mappings with $f\in L^{p}$ for some $p>1$. Laurain and Rivi\`ere \cite[Theorem 3.3]{Laurain-Riviere-2013-ACV}  also obtained $W^{4,p}$ regularity for \eqref{eq:inhomogenous Lamm-Riviere system}, but under special  growth conditions $$\begin{aligned}
|V| & \leq C|\nabla u| \\
|F| & \leq C|\nabla u|\left(\left|\nabla^{2} u\right|+|\nabla u|^{2}\right) \text { almost everywhere} \\
|w|+|\omega| & \leq C\left(\left|\nabla^{2} u\right|+|\nabla u|^{2}\right).
\end{aligned}$$

\begin{remark}\label{rmk:on W4p regularity}
The third order  regularity is  optimal for system  \eqref{eq:inhomogenous Lamm-Riviere system}.    In general, there does not exist $W^{4,p}$-regularity even for  homogenuous Lamm-Rivi\`ere system; see Example \ref{example:no W4p estimate} below.
\end{remark}

%Immediately we infer from the above two theorems some corollaries,
As immediate consequences of the above two theorems, we have the following   two corollaries, which might be of independent interest.
\begin{corollary}\label{coro:for system with critical growth}
 Let $u\in W^{2,2}(B_{1})$, $B_1\subset \R^4$, be a weak solution of
 \[\De^2 u=Q(x, u,\na u)+f,\quad  f\in L^p(B_1) \]
 with critical growth
\[|Q(x, u, \na u)|\le C|\na u|^4. \]
Then, $u\in W^{4,p}_{\loc}(B_1)$ if $p<3/4$ and $u\in W^{4,p-\de}_{\loc}(B_1)$ for any $\de>0$ if $p\ge 4/3$.
\end{corollary}
\begin{proof}
  Take $V=0, w=0, \om=0$ and $F=|\na u|^{-2} Q(x,u,\na u)\na u$.
\end{proof}
We remark that the homogeneous case of Corollary \ref{coro:for system with critical growth} has been studied in \cite{Strzelecki-Goldstein-2008-4thOrderPDE}.

%{\color{red} We have no examples to show that our estimates are optimal! Maybe we need to construct biharmonic mappings by scaling, as is done by Sharp-Topping.}

%As a first application of Theorem \ref{thm:optimal global estimate for inho Lamm-Riviere}, we obtain the following corollary.

\begin{corollary}[Energy gap]\label{coro:energy gap}
	Let $u\in W^{2,2}(\R^4,\R^m)$ be a weak solution of \eqref{eq:inhomogenous Lamm-Riviere system} in $\R^4$ with $V\in W^{1,2}(\R^4,M_m\otimes \Lambda^1\R^{4})$, $w\in L^{2}(\R^4,M_m)$, $\om\in L^{2}(\R^4,so_m)$, $F\in L^{\frac{4}{3},1}(\R^4,M_m\otimes \Lambda^1\R^{4})$ and $f\equiv 0$. Then there exists some $\ep=\ep(m)>0$ such that if
$$\|V\|_{W^{1,2}(\R^4)}+\|w\|_{L^{2}(\R^4)}
+\|\om\|_{L^{2}(\R^4)}+\|F\|_{L^{4/3,1}(\R^4)}<\ep,$$
 then $u\equiv 0$ in $\R^4$.
\end{corollary}

We can also derive a compactness result under the assumption that $f\in L^p$ for some $p>1$. However, for further applications,  we would like to weaken the $L^p$ integrability assumption on $f$ by assuming $f\in L\log L(B_{10})$. Note that under this weaker integrability assumption, we cannot get an effective decay estimate (as in Theorem \ref{thm:optimal Holder exponent for inho Lamm-Riviere}) to conclude H\"older continuity of weak solutions of \eqref{eq:inhomogenous Lamm-Riviere system}. However, we do have the following version of Theorem \ref{thm:optimal global estimate for inho Lamm-Riviere} corresponding to the boarderline case $p=1$, which can be regarded as a natural extension of \cite[Theorem 1.6]{Sharp-Topping-2013-TAMS} to the fourth order Lamm-Rivi\`ere system.
\begin{theorem}\label{thm:optimal local estimate LlogL}
Let $u\in W^{2,2}(B_{10},\R^m)$ be a weak solution of \eqref{eq:inhomogenous Lamm-Riviere system} with $f\in L\log L(B_{10},\R^m)$. Then, there exist some $\ep=\ep(m)>0$ and $C=C(m)>0$ such that if the smallness condition \eqref{eq:smallness assumption} is satisfied with $\ep_m=\ep$, then $$u\in W^{3,\frac{4}{3},1}_{\loc}(B_{10})$$ with
 \begin{equation}\label{eq:optimal LlogL}
 	\|u\|_{W^{3,\frac{4}{3},1}(B_{\frac{1}{2}})}\le C\left(\|f\|_{L\log L(B_{1})}+\|u\|_{L^{1}(B_{1})}\right).
 \end{equation}
 In particular, $u$ is continuous by the Sobolev embedding $W^{3,\frac{4}{3},1}(B_{10})\subset C(B_{10})$.
\end{theorem}

As an application of Theorem \ref{thm:optimal local estimate LlogL}, we obtain the following compactness result; compare it with \cite[Theorem 1.3]{Guo-Xiang-2019-Boundary}.
\begin{theorem}\label{thm:compactness}
Let $\{u_{n}\}_n\subset W^{2,2}(B_{10},\R^m)$ be a sequence of weak solutions of
\begin{eqnarray*}\label{}
\De^{2}u_n=\De(V_n\cdot\na u_n)+{\rm div}(w_n\na u_n)+(\na\om_n+F_n)\cdot\na u_n+f_n. &  &
\end{eqnarray*}
Suppose there exists a constant $\La>0$ such that
\[\sup_n \left(\|f_n\|_{L\log L(B_{10})}+\|u_n\|_{L^{1}(B_{10})}\right)\le \Lambda.\]
Then, there exist some $\ep=\ep(m)>0$  and  a mapping $u\in W^{2,2}(B_{10}, \R^m)$ such that, if the sequences $\{V_n,w_n,\om_n, F_n\}_{n\in \N}$ satisfy \eqref{eq:smallness assumption} with a common $\ep_m=\ep$, then after passing by to a subsequence,
\begin{eqnarray*}\label{}
 u_{n}\to u&  & \text{in } W_{\loc}^{2,2}(B_{10},\R^m).
\end{eqnarray*}
\end{theorem}
\subsection{Strategy of the proof}

Before ending this section, we would like to make some comments on the techniques that we shall use in the proofs of Theorems \ref{thm:optimal Holder exponent for inho Lamm-Riviere} and \ref{thm:optimal global estimate for inho Lamm-Riviere}. We follow the scheme of Sharp and Topping \cite{Sharp-Topping-2013-TAMS} and devide  the proofs into three steps.
\begin{enumerate}
  \item In the first step, we  derive the  decay estimates in Theorem \ref{thm:optimal Holder exponent for inho Lamm-Riviere} via the conservation law and  Hodge decomposition, from which H\"older continuity follows.  The decay estimates show that $\na^2 u$ and $\na u$ belong to some Morrey spaces.
  \item In the second step, we combine the above fact, together with the Riesz potential theory of Adams \cite{Adams-1975} (see Lemma \ref{lemma:improved Riesz potential}), to deduce an almost optimal higher order Soboelv regularity. A bootstrapping  argument is also applied.
  \item In the last step, we show that all the concerned  local estimates are uniform with respect to the parameters so that we can pass to the limit to conclude the optimal higher order Sobolev regularity with the critical exponent.
\end{enumerate}

As we are dealing with fourth order equation, the situation becomes more complicated than that of second order equation, especially in the first  and the last step,
  which are quite different from that of Sharp and Topping \cite{Sharp-Topping-2013-TAMS}.
The first step was the key step in their proofs. They adopted a very delicate iteration argument to derive sharp decay estimates. In order to run the iteration procedure,   a very precise control on  some coefficients of related  estimates is needed so that the key coefficients are sufficiently small. A fact that plays a crucial role in their proof is the (nondecreasing) monotonicity of the average function
 \[r\mapsto \frac{1}{r^2}\int_{B_r(x)} f\]
whenever $f$ is a subharmonic function.
  Such a monotonicity property seems  unknown in higher order cases. In particular, for a biharmonic function $h$ in $\R^4$,  we do not know whether the average function $r\mapsto \frac{1}{r^4}\int_{B_r(x)} |h|^k$ ($k\ge 1$) is monotone or not. Furthermore, we have to consider not only the energy of $\nabla u$, but also the energy of $\Delta u$. Thus,   it seems impossible to gain optimal coefficients simultaneously in front of the decay of these two energies.  Consequently, the iteration procedure of Sharp and Topping \cite[Proof of Lemma 7.3]{Sharp-Topping-2013-TAMS} fails to apply here. To overcome this difficulty,
we borrow some ideas from \cite{Wang-Zheng-2012-JFA} and \cite{Guo-Xiang-2019-Boundary}, and use Hodge decomposition together with Lorentz-Sobolev embedding to derive the optimal decay. %To better explain our approach,
To emphasize the differences between these two approaches, we included an alternative proof of the decay estimate of Sharp-Topping for the inhomogeneous system \eqref{eq:inhomogenous Lamm-Riviere system} in Section \ref{sec:reprove Sharp-Topping}.

Another severe technical difficulty occurs in the last step. In the case of Sharp-Topping, they run a similar scaling and iteration argument as in the first step due to the well control of coefficient of decay of energy of $\nabla u$. Moreover, since in the second order case they have proved that every solution belongs to $W^{2,p}_{\loc}$, the equation $-\De u=\Om\cdot \na u$ can be used as a pointwise identity to deduce   estimate for $\na^2 u$ directly. In the fourth order case, we cannot run a similar iteration argument as we do not have good enough control on the coefficients of the decay of energies of $\nabla u$ and $\nabla^2 u$. Furthermore, we have no fourth order regularity for the solutions, and thus equation only has the meaning of distributions.  To overcome these difficulties,  we apply a duality trick to obtain good control on the decay of energy of $\nabla^2 u$, and then apply the similar scaling and iteration procedure as Sharp and Topping. During the process, we will need several different  uniform estimates, in  which  the usually used elliptic regularity theory has to be refined. These difficulty makes the problem become more interesting.

%In a future coming work, we shall apply the regularity theorems to establish energy identity for sequences of weak solutions of the inhomogenuous Lamm-Rivi\`ere system \eqref{eq:inhomogenous Lamm-Riviere system}.
%We also expect that our method can be extended to even higher order systems such as the poly-harmonic type systems deduced recently by \textbf{the two German's work}.

%but in a different order. In their first step, they directly derive the almost optimal Sobolev regularity for weak solutions (without the use of optimal decay). Then they quantify the almost optimal second order Sobolev regularity in the second step. In the last step, they use these two estimates, together with a very clever iteration procedure, to produce the optimal decay estimate (see Theorem \ref{thm:Sharp-Topping} below).

This paper is organized as follows. Section \ref{sec: preliminaries} contain some preliminaries and auxiliary results for later proofs. In Section \ref{sec:reprove Sharp-Topping}, we present an alternative proof of a key result of Sharp-Topping \cite{Sharp-Topping-2013-TAMS}, which leads to \cite[Theorem 1.1]{Sharp-Topping-2013-TAMS}. Our main theorems are proved in Sections \ref{sec:Decay estimate forth system}, \ref{sec:higher order Sobolev regularity} and \ref{sec:optimal global estimates}. In the final section, Section \ref{sec:compactness}, we prove the compactness theorem. We also add two appendices to include certain auxiliary results that was used in the proofs of our main theorems.

Our notations are standard.
By $A\lesssim B$ we mean there exists a universal constant $C>0$ such that $A\le CB$.

\section{Preliminaries and auxiliary results}\label{sec: preliminaries}

\subsection{Function spaces and related}\label{sec:relevant function spaces}
Let $\Om\subset\R^{n}$ be a bounded smooth  domain, $1\le p<\wq$
and $0\le s<n$. The \emph{Morrey space} $M^{p,s}(\Om)$ consists
of functions $f\in L^{p}(\Om)$ such that
\[
\|f\|_{M^{p,s}(\Om)}\equiv\left(\sup_{x\in\Om,r>0}r^{-s}\int_{B_{r}(x)\cap\Om}|f|^{p}\right)^{1/p}<\wq.
\]
Denote by $L_*^p$  the weak $L^p$ space and define the \emph{weak Morrey space} $M^{p,s}_*(\Omega)$ as the space of functions $f\in L^p_*(\Omega)$ such that
$$\|f\|_{M^{p,s}_*(\Om)}\equiv \left(\sup_{x\in\Om,r>0}r^{-s}\|f\|^p_{L^p_{\ast}(B_r(x)\cap \Omega)}\right)^{1/p}<\wq,$$
where
$$\|f\|^p_{L^p_{\ast}(B_r(x)\cap \Omega)}\equiv\sup_{t>0}t^p\Big|\big\{x\in B_r(x)\cap \Omega: |f(x)|> t\big \}\Big|.$$

For a measurable function $f\colon \Om\to\R$, denote by $\de_{f}(t)=|\{x\in\Om:|f(x)|>t\}|$
its distributional function and by $f^{\ast}(t)=\inf\{s>0:\de_{f}(s)\le t\}$,
$t\ge0$, the nonincreasing rearrangement of $|f|$. Define
\begin{eqnarray*}
	f^{\ast\ast}(t)\equiv\frac{1}{t}\int_{0}^{t}f^{\ast}(s)\D s, &  & t>0.
\end{eqnarray*}
The\emph{ Lorentz space} $L^{p,q}(\Om)$ ($1<p<\wq,1\le q\le\wq$)
is the space of measurable functions $f:\Om\to\R$ such that
\[
\|f\|_{L^{p,q}(\Om)}\equiv\begin{cases}
\left(\int_{0}^{\wq}(t^{1/p}f^{\ast\ast}(t))^{q}\frac{\D t}{t}\right)^{1/q}, & \text{if }1\le q<\wq,\\
\sup_{t>0}t^{1/p}f^{\ast\ast}(t) & \text{if }q=\wq
\end{cases}
\]
is finite.

%Let $1\le p<\wq$ and $1\le q\le\wq$, $0<s<n$.
%%Denote by $L^{p,q}(\Om)$ the classical Lorentz space.
%The \emph{Lorentz-Morrey space} $LM^{p,q,s}(\Om)$ consists of functions $u\in L^{p,q}(\Om)$ such that
%\[
%\|u\|_{LM^{p,q,s}(\Om)}^{p}\equiv\sup_{x\in\Om,r>0}r^{-s}\|u\|_{L^{p,q}(B_{r}(x)\cap\Om)}^{p}.
%\]
%It is easy to see that $LM^{p,p,s}=M^{p,s}$, $LM^{p,\infty,s}=M_{*}^{p,s}$ and since $L^{p,q_{1}}\subset L^{p,q_{2}}$
%if $q_{1}<q_{2}$, we have
%\[
%LM^{p,q_{1},s}\subset LM^{p,q_{2},s}.
%\]
%It is straightforward to verify that for each $k\in \N$, we have
%$$\|u^k\|_{LM^{p,q,s}}=\|u\|_{LM^{kp,kq,s}}.$$

  It is well-known that $L^{p,p}=L^p$ and $L^{p,\wq}=L^p_{\ast}.$ We will need the following H\"older's inequality in Lorentz spaces.

\begin{proposition} \label{prop: Lorentz-Holder inequality} (\cite{ONeil-1963})
	Let $1<p_{1},p_{2}<\wq$ and $1\le q_{1},q_{2}\le\wq$ be such that
	\begin{eqnarray*}
		\frac{1}{p}=\frac{1}{p_{1}}+\frac{1}{p_{2}}\le1 & \text{and} & \frac{1}{q}=\frac{1}{q_{1}}+\frac{1}{q_{2}}\le1.
	\end{eqnarray*}
	Then, $f\in L^{p_{1},q_{1}}(\Om)$ and $g\in L^{p_{2},q_{2}}(\Om)$
	implies $fg\in L^{p,q}(\Om)$. Moreover,
	\[
	\|fg\|_{L^{p,q}(\Om)}\le\|f\|_{L^{p_{1},q_{1}}(\Om)}\|g\|_{L^{p_{2},q_{2}}(\Om)}.
	\]
\end{proposition}

The first order \emph{Lorentz-Sobolev space} $W^{1,p,q}(\Om)$ for $1<p<\wq,1\le q\le\wq$
 consists of functions $f\in L^{p,q}(\Om)$  with weak gradient $\na f\in L^{p,q}(\Om)$.
  A natural norm for a Lorentz-Sobolev function $f\in W^{1,p,q}(\Om)$ is defined  by
\[
\|f\|_{W^{1,p,q}(\Om)}=\left(\|f\|_{L^{p,q}(\Om)}^{p}+\|\na f\|_{L^{p,q}(\Om)}^{p}\right)^{1/p}.
\]
Higher order Lorentz-Sobolev spaces can be defined analogously.
\begin{proposition}(\cite{Tartar-1998})\label{prop:Lorentz-Sobolev embedding}  Let $1<p<\wq$ and $1\le q\le \wq$.  Then
	
	\begin{itemize}
		\item[(1).] $W^{1,p}(\Om)=W^{1,p,p}(\Om)$;
		\item[(2).] If $\Omega$ is bounded and smooth, then $W^{1,p,q}(\Om)$ embeds continuously into $L^{p^{*},q}(\Om)$
		for $1<p<n$, where $1/p^{*}=1/p-1/n$.
         \item[(3).]  $W^{1,n,1}(\Om)\subset C(\Om)$.
	\end{itemize}
\end{proposition}

We shall need the space $L\log L$ as well. Recall that $L\log L(\Omega)$ consists of all functions $f\colon \Omega\to \R$ such that
\[\|f\|_{L\log L(\Omega)}:=\int_{0}^\infty f^\ast(t)\log\Big(2+\frac{1}{t}\Big)dt<\infty.\]
The following elementary fact on functions in $L\log L$ can be found in \cite[Lemma 2.1]{Sharp-Topping-2013-TAMS}.

\begin{lemma}\label{lemma:ST 2.1}
Suppose $f\in L\log L(B_r)$ and $r\in (0,\frac{1}{2})$. Then there exists a constant $C>0$, independent of $r$, such that
\[\|f\|_{L^1(B_r)}\leq C\Big[\log(\frac{1}{r}) \Big]^{-1}\|f\|_{L\log L(B_r)}.\]
\end{lemma}

%As a consequence of Proposition \ref{prop: Lorentz-Holder inequality}, we infer that if $f\in LM^{p_{1},q_{1},s}(\Om)$
%and $g\in LM^{p_{2},q_{2},s}(\Om)$, then $fg\in{\rm LM}^{p,q,s}(\Om)$,
%and
%\begin{equation}
%\|fg\|_{LM^{p,q,s}(\Om)}\le\|f\|_{LM^{p_{1},q_{1},s}(\Om)}\|g\|_{LM^{p_{2},q_{2},s}(\Om)}..\label{eq: Multiplicity rule}
%\end{equation}
%Indeed, let $x\in\Om$ and $r>0$. Then,
%\[
%\|fg\|_{L^{p,q}(\Om\cap B_{r}(x))}\le\|f\|_{L^{p_{1},q_{1}}(\Om\cap B_{r}(x))}\|g\|_{L^{p_{2},q_{2}}(\Om\cap B_{r}(x))}.
%\]
%Thus,
%\[
%r^{-s/p}\|fg\|_{L^{p,q}(\Om\cap B_{r}(x))}\le\left(r^{-s/p_{1}}\|f\|_{L^{p_{1},q_{1}}(\Om\cap B_{r}(x))}\right)\left(r^{-s/p_{2}}\|g\|_{L^{p_{2},q_{2}}(\Om\cap B_{r}(x))}\right).
%\]
%Taking supremum over $x$ and $r$ gives (\ref{eq: Multiplicity rule}).
%
%A special case of (\ref{eq: Multiplicity rule}) will be repeatedly applied in our later proofs:
%$p_{1}=p_{2}=4$ and $q_{1}=2,q_{2}=\wq$, then $p=q=2$. More precisely, we shall apply
%%\[
%%LM^{4,2,s}\cdot LM^{4,\wq,s}\subset LM^{2,2,s}=M^{2,s}.
%%\]
%\begin{equation}\label{eq:easy holder LM weak}
%\|fg\|_{M^{2,s}}\leq \|f\|_{LM^{4,2,s}}\|g\|_{LM^{4,\infty,s}}=\|f\|_{LM^{4,2,s}}\|g\|_{M_*^{4,s}}.
%\end{equation}

\subsection{Fractional Riesz operators}

Let $0<\al<n$ and  $I_\alpha=c_{n,\al}|x|^{\alpha-n}$, $x\in \R^n$, be the usual fractional Riesz operators, where $c_{n,\al}$ is a positive normalization constant. The following well-known estimates on fractional Riesz operators in Morrey spaces were proved by Adams \cite{Adams-1975}.

\begin{proposition}\label{prop:Adams 1975} Let $0< \alpha<n$, $0\le \lambda< n$ and $1\le p<\frac{n-\lambda}{\alpha}$. Then, there exists a constant $C>0$ depending only $n,\al,\la$ and $p$ such that, for all $f\in M^{p,\lambda}(\R^n)$, there holds

	(i) If  $p>1$, then
 \begin{equation}\label{eq:Riesz Adams 1}
		\|I_\alpha(f)\|_{M^{\frac{(n-\la)p}{n-\la-\al p},\lambda}(\R^n)}\leq C\|f\|_{M^{p,\lambda}(\R^n)}.
		\end{equation}
					
		(ii) If $p=1$, then
		\begin{equation}\label{eq:Riesz Adams 2}
		\|I_\alpha(f)\|_{M_{*}^{\frac{n-\lambda}{n-\lambda-\alpha},\lambda}(\R^n)}\leq C\|f\|_{M^{1,\lambda}(\R^n)}.
		\end{equation}	
\end{proposition}
Note that when $\la=0$ we recover  the usual theory of  Riesz potentials between $L^p$ spaces.

%\textbf{(The following fractional maximal operator seems unnecessary)} Define for any $0\le\be\le n$, \[ M_{\be}f(x)=\sup_{r>0}r^{\be-n}\int_{B_{r}(x)}|f|.\]
%Note that $M_{0}f=Mf$ and $M_{n}f=\|f\|_{L^{1}}$. In case $f\in L_{\loc}^{1}(\Om)$,
%we may define \[ M_{\be,\Om}f(x)=\sup_{r>0}r^{\be-n}\int_{B_{r}(x)\cap\Om}|f|. \]Then $f\in M^{1,n-\be}(\Om)$ if and only if $M_{\be,\Om}f\in L^{\wq}(\Om)$.

%\begin{lemma}(Adams' inequality) Let $0<\al<\be\le n$. Then, for
%	any $f\in L_{\loc}^{1}(\R^{n})$, there holds
%	\[
%	M_{\al}f(x)\le I_{\al}|f|(x)\le C_{n,\al,\be}(M_{\be}f(x))^{\al/\be}(Mf(x))^{1-\al/\be}.
%	\]
%\end{lemma}

%This implies \textbf{the improve Riesz potential theory} as that of
%Lemma A.3 of Sharp-Topping.

The following lemma can be viewed as an improved version of the classical Riesz potential theory. It plays a key role in the proofs of Sharp-Topping \cite{Sharp-Topping-2013-TAMS}, and  also in our approach.

\begin{lemma}[\cite{Adams-1975}, Proposition 3.1]\label{lemma:improved Riesz potential}
	Let $0<\al<\be\le n$ and $f\in M^{1,n-\be}(\R^{n})\cap L^{p}(\R^{n})$
	for some $1<p<\wq$. Then, $I_{\al}f\in L^{\frac{\be p}{\be-\al}}(\R^{n})$
	with
	\[
	\|I_{\al}f\|_{\frac{p\be}{\be-\al}, \R^n}\le C_{\al,\be,n,p}\|f\|_{M^{1,n-\be}(\R^n)}^{\frac{\al}{\be}}\|f\|_{p, \R^n}^{\frac{\be-\al}{\be}}. %\le C\|f\|_{p}^{\frac{\be-\al}{\be}}.
	\]
	\end{lemma}
%\begin{proof}
%	By Adams' inequality,
%	\[
%	|I_{\al}f(x)|\le C(M_{\be}f)^{\al/\be}(Mf(x))^{1-\al/\be}.
%	\]
%	The corollary is a direct consequence of this inequality.
%\end{proof}
In particular, in the case $1<p<n/\be$, we have
	\[
	\frac{\be p}{\be-\al}>\frac{np}{n-\al p},
	\]
	which implies that $I_{\al}f$ has better integrability than the typical one from $L^p$ boundedness. This improved Riesz potential estimate comes from the fact that $f$ has additional fine property, that is, $f$ also belongs to some
Morrey space. Lemma A.3 of \cite{Sharp-Topping-2013-TAMS} gives a local version of Lemma \ref{lemma:improved Riesz potential}.

%\textbf{Morrey space} is defined as follows:
%
%Let $\Om\subset\R^{n}$ be an open set with smooth boundary. Let $1\le p<\wq$
%and $0\le s<n$. The homogeneous Morrey space $M^{p,s}(\Om)$ consists
%of functions $f\in L^{p}(\Om)$ such that
%\[
%\|f\|_{M^{p,s}(\Om)}\equiv\left(\sup_{x\in\Om,r>0}r^{-s}\int_{B_{r}(x)\cap\Om}|f|^{p}\right)^{1/p}<\wq.
%\]
%$k$th order Morrey space $M_{k}^{p,n-kp}(\Om)$ consists of $f\in W^{k,p}(\Om)$
%such that $\na^{l}u\in M^{p,k-lp}(\Om)$. \textbf{Note that Struwe
%	used the notation $L_{k}^{p,m-kp}$ to denote $k$th order Morrey
%	space, but we used the symbol $M_{k}^{p,n-kp}$. The letter $L$ will
%	be mostly related to Lorentz spaces later. }
%
%\textbf{Lorentz-Morrey space.}
%
%Let $1\le p<\wq$ and $1\le q\le\wq$, $0<s<n$. Use $L^{p,q}(\Om)$
%to denote Lorentz spaces. The \emph{Lorentz-Morrey space} $LM^{p,q,s}(\Om)$
%consists of functions $u\in L^{p,q}(\Om)$ such that
%\[
%\|u\|_{LM^{p,q,s}(\Om)}^{p}\equiv\sup_{x\in\Om,r>0}r^{-s}\|u\|_{L^{p,q}(B_{r}(x)\cap\Om)}^{p}.
%\]
%Easy to see that $LM^{p,p,s}=M^{p,s}$, and since $L^{p,q_{1}}\subset L^{p,q_{2}}$
%if $q_{1}<q_{2}$, we have
%\[
%LM^{p,q_{1},s}\subset LM^{p,q_{2},s}.
%\]

\subsection{Scaling invariance of \eqref{eq:inhomogenous Lamm-Riviere system}}
We shall  use a scaling argument in our later proofs. Let $u$ be a weak solution of \eqref{eq:inhomogenous Lamm-Riviere system}. % and assume the smallness condition \eqref{eq:smallness assumption} is satisfied. Then there exist $A\in W^{2,2}\cap L^{\wq}(B_{8},M(m))$ and $B\in L^{2}(B_{8},M(m)\otimes\wedge^{2}\R^{4})$ such that \eqref{eq:condition for A and B} is satisfied. Consequently, we obtain the following conservation law for $u$:
% \begin{eqnarray}\label{eq: Conservation law inhomogeneous} \De(A\De u)={\rm div}K+Af &  & \text{in }B_{8}, \end{eqnarray}
%with $K$ given by \eqref{eq: K}. We shall frequently use this conservation law in our later proofs.
%\begin{equation}\label{eq: K}
%K=2\na A\De u-\De A\na u-Aw\na u-\na AV\cdot\na u+A\na(V\cdot\na u)+B\cdot\na u.
%\end{equation}
For any  $B_{2R}(x_{0})\subset B_{10}$ and $x\in B_{2}=B_{2}(0)$, set
\[
\begin{aligned} & u_{R}(x)=u(x_{0}+Rx),\\
& V_{R}(x)=RV(x_{0}+Rx),\quad w_{R}(x)=R^{2}w(x_{0}+Rx),\\
& \om_{R}(x)=R^{2}\om(x_{0}+Rx),\quad F_{R}(x)=R^{3}F(x_{0}+Rx),\quad f_{R}(x)=R^{4}f(x_{0}+Rx).
\end{aligned}
\]
It is straightforward to verify that $u_R$ satisfies
\begin{eqnarray}\label{eq: scaled equation}
\De^{2}u_{R}=\De(V_{R}\cdot\na u_{R})+{\rm div}(w_{R}\na u_{R})+(\na\om_{R}+F_R)\cdot\na u_{R}+f_{R} &  & \text{in }B_{2}.
\end{eqnarray}
Moreover, for any $0<r<1$ and $1\le q\le \wq$,  there holds
\[
\begin{aligned} & \|\na u_{R}\|_{L^{4,q}(B_{r}(0))}=\|\na u\|_{L^{4,q}(B_{rR}(x_{0}))},\quad\|\De u_{R}\|_{L^{2}(B_{r}(0))}=\|\De u\|_{L^{2}(B_{rR}(x_{0}))},\\
 & \|V_{R}\|_{L^{4,q}(B_{r}(0))}=\|V\|_{L^{4,q}(B_{rR}(x_{0}))},\quad\|\na V_{R}\|_{L^{2}(B_{r}(0))}=\|\na V\|_{L^{2}(B_{rR}(x_{0}))},\\
 & \|w_{R}\|_{L^{2}(B_{r}(0))}=\|w\|_{L^{2}(B_{rR}(x_{0}))},\\
 & \|\om_{R}\|_{L^{2}(B_{r}(0))}=\|\om\|_{L^{2}(B_{rR}(x_{0}))},\quad\|F_{R}\|_{L^{4/3,q}(B_{r}(0))}=\|F\|_{L^{4/3,q}(B_{rR}(x_{0}))},\\
 & \|f_{R}\|_{L^{p}(B_{r}(0))}=R^{4(1-1/p)}\|f\|_{L^{p}(B_{rR}(x_{0}))},
\end{aligned}
\]
and \[\|f(x_0 +R\cdot )\|_{L\log L(B_1)}\leq C\|f\|_{L\log L(B_R(x_0))}.\]
by \cite[Lemma 2.2]{Sharp-Topping-2013-TAMS} whenever $B_R(x_0)\subset B_{10}$.

\section{Warm up}\label{sec:reprove Sharp-Topping}

 Let $u\in W^{1,2}(B_{10},\R^m)$, $B_{10}\subset \R^2$, be a weak solution of system \eqref{eq:inhomogenuous system ST} and $f\in L^p(B_{10},\R^m)$.
Sharp and Topping \cite[Lemma 7.3]{Sharp-Topping-2013-TAMS} proved that if $p\in (1,2)$, then
\begin{eqnarray*}\label{eq: decay of S-T}
	\begin{aligned}
	\|\na u\|_{L^{2}(B_r)}\leq  Cr^{\alpha}\left(\|\nabla u\|_{L^{2}(B_{1})}+\|f\|_{L^p(B_1)}\right)
	\end{aligned}
	\end{eqnarray*}
for $0<r<1$ under a smallness assumption on $\Om$, where $\al=2(1-\frac{1}{p})$, from which the local $\al$-H\"older continuity follows. The method there is quite tricky and requires a very delicate control on coefficients of various inequalities throughout their  arguments.

The aim of this section is to reproduce the above decay estimate by refining the technique of \cite[Lemma 7.3]{Sharp-Topping-2013-TAMS}. The refined technique will be applied to the fourth order system \eqref{eq:inhomogenous Lamm-Riviere system} in the next section, but in a more complexed way.

\begin{theorem}\label{thm:Sharp-Topping}
	Let $u\in W^{1,2}(B_{10},\R^m)$ be a weak solution to \eqref{eq:inhomogenuous system ST}. Set $\alpha=2\big(1-\frac{1}{p} \big)$ if $1<p<2$ and $\alpha$ to be any number in $(0,1)$ if $p\geq 2$. Then there exist constants $\ep=\ep(p,m)>0$ and $C=C(p,m,\alpha)$, such that if $\|\Omega\|_{L^2(B_{10})}\leq \ep$, then
	\begin{eqnarray}\label{eq: decay via CL}
	\begin{aligned}
	\|\na u\|_{L^{2,1}(B_\ga)}\leq  C\ga^{\alpha}\left(\|\nabla u\|_{L^{2,1}(B_{1})}+\|f\|_{L^p(B_1)}\right).
	\end{aligned}
	\end{eqnarray}
	\end{theorem}

\begin{proof}
Choose $\ep$  so small that there exist $A\in W^{1,2}\cap L^\infty$ and $B\in W^{1,2}$ such that \eqref{eq:inhomogenuous system ST} can be written as
\begin{equation}\label{eq:inhomogenuous system ST via conservation law}
-\divergence(A\nabla u)=\nabla^\perp B\cdot \nabla u+Af.
\end{equation}

Next  extend all the  functions from $B_1$ to the whole space $\R^2$ in such a way that their norms in $\R^2$ are bounded  by a constant multiply of the corresponding norms in $B_1$.  With no confuse of notations, we use the same symbols for all the extended functions.

Applying the Hodge decomposition for $A\nabla u$, we obtain
$$A\D u=\D r+*\D g\qquad \text{in } \R^2.$$
Let $\Gamma(x)=-\frac{1}{2\pi}\log |x|$ be the fundamental solution of $-\Delta$ in $\R^2$. Set $r_1=\Gamma\ast \big(\nabla^\perp B\cdot \nabla u \big)$ and $r_2=\Gamma\ast (Af)$ in $\R^2$.
It follows that
\begin{eqnarray*}
  -\De (r-r_1-r_2)=0& \text{in } B_1.
\end{eqnarray*}
Thus we obtain
\begin{eqnarray*}
& A\D u=\D r_1+\D r_2+\ast \D g+ h & \text{in } B_1
\end{eqnarray*}
for some harmonic 1-form in $B_1$.

We estimate each term in the above decomposition as follows: First note that $\nabla^\perp B\cdot \nabla u$ belongs to the Hardy space $\mathcal{H}^1(\R^2)$. So
\[
\begin{aligned}
\|\nabla {r_1}\|_{L^{2,1}(\R^2)} & \lesssim \|\nabla\Gamma\ast \big(\nabla^\perp B\cdot \nabla u \big)\|_{L^{2,1}(\R^2)}\lesssim \|\nabla^\perp B\cdot \nabla u\|_{\mathcal{H}^1(\R^2)}\\
& \lesssim \|B\|_{L^2(\R^2)}\|\nabla u\|_{L^2(\R^2)}\lesssim \ep \|\nabla u\|_{L^2(B_1)},\\
\end{aligned}
\]
where in the second inequality above we used the fact that $\nabla\Gamma\colon \mathcal{H}^1(\R^2)\to L^{2,1}(\R^2)$ is bounded (see e.g.~\cite[Section A.4]{Sharp-Topping-2013-TAMS}). Similarly, we have
\[
\begin{aligned}
\|\nabla g\|_{L^{2,1}(\R^2)} \lesssim \|A\|_{L^2(\R^2)}\|\nabla u\|_{L^2(\R^2)}\lesssim \ep \|\nabla u\|_{L^2(B_1)}.
\end{aligned}
\]
Since $\nabla^2\Gamma$ is a singular operator,
$$\|\nabla^2 r_2\|_{L^{p}(\R^2)}\lesssim \|f\|_{L^p(B_1)}.$$
Set $\bar{p}=\frac{2p}{2-p}$ for $1<p<2$ and any finite number if $p\geq 2$. When $1<p<2$, applying the Lorentz-Sobolev embedding (see Proposition \ref{prop:Lorentz-Sobolev embedding}), we have
$$\|\nabla r_2\|_{L^{\bar{p},p}(\R^2)}\lesssim \|f\|_{L^p(B_1)}.$$
In this case,  H\"older's inequality gives
\[
\|\nabla r_2\|_{L^{2,1}(B_{\ga})}\lesssim \ga^{2(1-\frac{1}{p})}\|\nabla r_2\|_{L^{\bar{p},p}(B_{\ga})}\lesssim \ga^{2(1-\frac{1}{p})}\|f\|_{L^p(B_1)}
\]
for any $\ga\in(0,1)$. If $p\geq 2$, we may use H\"older's inequality and the Sobolev embedding for $p\geq 2$ to get
\[
\|\nabla v\|_{L^{2,1}(B_{\ga})}\lesssim \ga^{\alpha}\|f\|_{L^p(B_1)}
\]
for any $\ga\in(0,1)$.

Combining the above  estimates together, it follows that $\na u\in L^{2,1}(B_1).$
Moreover,  when $1<p<2$, we have
\[
\begin{aligned}\|\nabla u\|_{L^{2,1}(B_{\ga})} & \le\|h\|_{L^{2,1}(B_{\ga})}+\|\nabla {r_1}\|_{L^{2,1}(B_{\ga})}+\|\na g\|_{L^{2,1}(B_{\ga})}+\|\nabla r_2\|_{L^{2,1}(B_{\ga})}\\
& \lesssim \ga\|h\|_{L^{2,1}(B_{1})}+\|\nabla r_1\|_{L^{2,1}(B_{1})}+\|\na g\|_{L^{2,1}(B_{1})}+\ga^{2(1-\frac{1}{p})}\|\nabla r_2\|_{L^{\bar{p},p}(B_{1})}\\
& \lesssim \ga\|\na u\|_{L^{2}(B_{1})}+\ep_m\|\na u\|_{L^2(B_\ga)}+\ga^{2(1-\frac{1}{p})}\|f\|_{L^p(B_1)}\\
&\lesssim (\gamma+\ep_m)\|\nabla u\|_{L^{2,1}(B_1)}+\ga^{2(1-\frac{1}{p})}\|f\|_{L^p(B_1)},
\end{aligned}
\]
and similarly, when $p\geq 2$,
\begin{eqnarray}\label{eq: decay via CL p > 2}
\begin{aligned}
\|\na u\|_{L^{2,1}(B_\ga)} \lesssim (\gamma+\ep_m)\|\nabla u\|_{L^{2,1}(B_1)}+\ga^{\al}\|f\|_{L^p(B_1)}
\end{aligned}
\end{eqnarray}
for any $\al\in (0,1)$.

The last step is to iterate. Let $\ga>0$ to be determined and choose $\ga\geq \epsilon_m$ so that
\[
\|\na u\|_{L^{2,1}(B_{\ga})}\le C\left(\ga\|\na u\|_{L^{2,1}(B_{\ga})}+\ga^{\al}\|f\|_{L^{p}(B_{1})}\right).
\]
Note that the equation is scaling invariant: for any $\tau>0$,
the functions $u_{\tau}(x)=u(x_{0}+\tau x)$,  $f_{\tau}(x)=\tau^{2}f(x_{0}+\tau x)$ and
$\Om_{\tau}=\tau\Om(x_{0}+\tau x)$ satisfy
\[
-\De u_{\tau}=\Om_{\tau}\cdot\na u_{\tau}+f_{\tau}.
\]
Hence the sequence $\{a_n\}_{n\in \N}$, with $a_{n}=\|\na u\|_{L^{2,1}(B_{\ga^{n}})}$, satisfies
\[
a_{n}\le C\ga a_{n-1}+C\ga^{n\al}\|f\|_{L^{p}(B_{1})}.
\]
Iteration gives
\[
a_{n}\le(C\ga)^{n}a_{0}+\ga^{n\al}\left(\sum_{i=0}^{n-1}\left(C\ga^{1-\al}\right)^{i}\right)\|f\|_{L^{p}(B_{1})}.
\]
Choose $\ga$ such that $C\ga^{1-\al}<1$ and we achieve
\[
a_{n}\le C\ga^{n\al}(a_{0}+\|f\|_{L^{p}(B_{1})}).
\]
The proof is complete.
\end{proof}

%\begin{remark}\label{rmk:on the decay of Sharp-Topping in Lp scale}
	%It is clear from the above proof that, we may replace all the $L^{2,1}$-norms by the corresponding $L^2$-norms to arrive at the following decay estimate in $L^2$-scale,
%\begin{eqnarray}\label{eq: decay via CL in Lp}
%\begin{aligned} \|\na u\|_{L^{2}(B_\ga)}\leq  C\ga^{\alpha}\left(\|\nabla u\|_{L^{2}(B_{1})}+\|f\|_{L^p(B_1)}\right), \end{aligned}
%\end{eqnarray} where $\alpha=2\big(1-\frac{1}{p} \big)$ if $1<p<2$ and $\alpha$ can be any number in $(0,1)$ if $p\geq 2$.\end{remark}

\section{H\"older regularity via decay estimates}\label{sec:Decay estimate forth system}

In this section, we  prove Theorem \ref{thm:optimal Holder exponent for inho Lamm-Riviere}, which is a fourth order analog of Theorem \ref{thm:Sharp-Topping}.  The idea of the proof is quite similar to that used in Theorem \ref{thm:Sharp-Topping}, but more complicated.  %We will repeatedly use the smallness condition  \eqref{eq:smallness assumption}. That is, assume there exists some $\ep_m>0$ such that
 %\begin{equation}\label{eq:smallness assumption again}
%\|V\|_{W^{1,2}(B_{{10}})}+\|w\|_{L^{2}(B_{{10}})}
%+\|\om\|_{L^{2}(B_{10})}+\|F\|_{L^{4/3}(B_{10})}<\ep_{m}.
%\end{equation}

%\begin{theorem}\label{thm:optimal Decay estimates for Lamm-Riviere}
	%Let $u\in W^{2,2}(B_1,\R^m)$ be a weak solution to \eqref{eq:inhomogenous Lamm-Riviere system} and set $\alpha=4\big(1-\frac{1}{p} \big)$ if $1<p<4/3$ and $\alpha$ to be any number in $(0,1)$ if $p\geq \frac{4}{3}$. Then there exists $\ep>0$, depending only on $p$ and $m$, such that if the smallness condition
 %\eqref{eq:smallness assumption again} holds with $\ep_m=\ep$, then
 %   \begin{equation}\label{eq:Lamm-Riviere decay via CL}
  %  \begin{aligned}
  %  \|\na u\|_{L^{4,2}(B_\ga)}+\|\De u\|_{L^2(B_\ga)}\leq  C\ga^{\al}\left(\|\nabla u\|_{L^{4,2}(B_{1})}+\|\De u\|_{L^2(B_1)}+\|f\|_{L^p(B_1)}\right),
  %    \end{aligned}
   %  \end{equation}
	%where $C=C(p,m,\alpha)$ is a constant.
%\end{theorem}

%As an immediate consequence of Theorem \ref{thm:optimal Decay estimates for Lamm-Riviere}, we obtain the optimal H\"older exponent for weak solutions of \eqref{eq:inhomogenous Lamm-Riviere system}, which corresponds to the first assertion of Theorem \ref{thm:optimal Holder exponent for inho Lamm-Riviere}.

%\begin{corollary}\label{coro:optimal Holder continuity}
%	Under the assumption of Theorem \ref{thm:optimal Holder exponent for inho Lamm-Riviere}, we have $u$ is H\"older continuity with expenent $\alpha=4(1-\frac{1}{p})$.
%\end{corollary}

\begin{proof}[Proof of Theorem \ref{thm:optimal Holder exponent for inho Lamm-Riviere}]
We begin with the conservation law of Lamm and Rivi\'ere \cite{Lamm-Riviere-2008}. Take  $\ep_m>0$  sufficiently small such that  the  smallness condition  \eqref{eq:smallness assumption} holds. Then there exist $A\in W^{2,2}\cap L^{\wq}(B_{8},M(m))$ and $B\in L^{2}(B_{8},M(m)\otimes\wedge^{2}\R^{4})$ such that
\begin{eqnarray}\label{eq: Conservation law inhomogeneous again}
\De(A\De u)={\rm div}K+Af &  & \text{in }B_{8},
\end{eqnarray}
with $K$ given by \eqref{eq: K}.
  As in the previous section, we extend all the relevant functions from $B_1$ to $\R^4$ such that their norms in $\R^4$ are bounded  by a constant multiply of the corresponding norms in $B_1$.  To simplify the notation, we use the same symbols for all the extended functions. Below we first estimate the decay of $\|\De u\|_{L^2(B_r)}$, and then the decay of $\|\na u\|_{L^{4,2}(B_r)}$; finally we  combine the two decay estimates together to conclude the proof.

1. Decay of $\|\De u\|_{L^2(B_r)}$.

 By  an elementary computation using Proposition \ref{prop: Lorentz-Holder inequality} and \eqref{eq: A-B small}, we  obtain
\begin{equation}\label{eq: estimate of K}
 \|{K}\|_{L^{\frac{4}{3},1}(B_{1})}\lesssim \ep_m(\|\nabla u\|_{L^{4,2}(B_{1})}+\|\nabla^2 u\|_{L^2(B_{1})}).
\end{equation}
For instance, for the first term $\na A\De u$, we have
\[ \|\na A\De u\|_{L^{\frac{4}{3},1}(B_{1})}\le \|\na A\|_{L^{4,2}(B_{1})}\|\De u\|_{L^{2}(B_{1})}\lesssim \ep_m \|\na^2 u\|_{L^{2}(B_{1})}.\]
The rest terms are estimated similarly. For details, see e.g.  \cite[proof of Lemma 3.1]{Guo-Xiang-2019-Boundary}.

Next let $I_2$ be the fundamental solution of $-\De$ in $\R^4$ and set $u_1=I_2 \ast {\rm div}K$, $u_2=I_2\ast Af$ in $\R^4$. The theory of  singular integrals implies that
\[\|\na u_1\|_{L^{4/3,1}(\R^4)}\lesssim \|K\|_{L^{4/3,1}(\R^4)}\lesssim \|K\|_{L^{4/3,1}(B_1)}\]
and
\begin{equation}\label{eq: Lp theory of inhomo terms}
\|\na^2 u_2 \|_{L^p(\R^4)} \lesssim \|f \|_{L^p(\R^4)}\lesssim \|f \|_{L^p(B_1)}.
\end{equation}
Combining the embedding $W^{1,\frac 43, 1}(\R^4)\subset L^{2,1}(\R^4)$ and \eqref{eq: estimate of K}, we deduce
\begin{equation}\label{eq: estimate of u_1}
\begin{aligned}
\left\|u_{1}\right\|_{L^{2}\left(B_{1}\right)} & \lesssim\left\|u_{1}\right\|_{L^{2,1}\left(B_{1}\right)} \leqslant\left\|u_{1}\right\|_{L^{2,1}\left(\mathbb{R}^{4}\right)} \lesssim\left\|\nabla u_{1}\right\|_{L^{4/3,1}\left(\mathbb{R}^{4}\right)} \\
&  \lesssim\|K\|_{L^{\frac{4}{3}, 1}\left(B_{1}\right)}
\lesssim\ep_m(\|\nabla u\|_{L^{4,2}(B_{1})}+\|\nabla^2 u\|_{L^2(B_{1})}).
\end{aligned}
\end{equation}

Now it is easy to see that the function $v=A\De u-u_1-u_2$ is harmonic in $B_1$. Therefore, for any $0<\tau<1$,
\[\int_{B_{\tau}}|v|^{2} \leqslant C \tau^{4} \int_{B_{1}}|v|^{2}.\]
As a consequence,  for any $0<\tau<1$, it holds
\begin{equation}
\begin{aligned}\int_{B_{\tau}}|\Delta u|^{2} & \lesssim \int_{B_{\tau}}|v|^{2}+\int_{B_{\tau}}\left|u_{1}\right|^{2}
+\int_{B_{\tau}}|u_{2}|^{2}\\
 & \lesssim\tau^{4}\int_{B_{1}}|v|^{2}+\int_{B_{1}}\left|u_{1}\right|^{2}+\int_{B_{\tau}}|u_{2}|^{2}\\
 & \lesssim\tau^{4}\int_{B_{1}}|\Delta u|^{2}+\left(1+\tau^{4}\right)\int_{B_{1}}\left|u_{1}\right|^{2}+\int_{B_{\tau}}|u_{2}|^{2}+\tau^{4}\int_{B_{1}}|u_{2}|^{2}\\
 & \lesssim\left(\tau^{4}+\epsilon_{m}^{2}\right)\int_{B_{1}}|\na^{2}u|^{2}+\epsilon_{m}^{2}\|\nabla u\|_{L^{4,2}\left(B_{1}\right)}^{2}+\int_{B_{\tau}}|u_{2}|^{2}+\tau^{4}\int_{B_{1}}|u_{2}|^{2}\\
 & \lesssim\left(\tau^{4}+\epsilon_{m}^{2}\right)\int_{B_{2}}|\De u|^{2}+\epsilon_{m}^{2}\|\nabla u\|_{L^{4,2}\left(B_{2}\right)}^{2}+\int_{B_{\tau}}|u_{2}|^{2}
 +\tau^{4}\int_{B_{1}}|u_{2}|^{2}.
\end{aligned}
\label{eq: before decay estimate of Delta u}
\end{equation}
In the first line above we applied the fact that $|\De u|\approx |A\De u|$ holds since $|A-\Id|\le \ep_m$. In the last second line we applied the estimate  \eqref{eq: estimate of u_1} of $u_1$.
Continuing from the last line, we apply the interior $L^2$ estimate to derive
\begin{equation}\label{eq: interior L2 estimate}
\left\|\nabla^{2} u\right\|_{L^{2}\left(B_{1}\right)} \lesssim \|\Delta u\|_{L^{2}\left(B_{2}\right)}^{2}+\|\nabla u\|_{L^{4,2}\left(B_{2}\right)}^{2}.
\end{equation}
by assuming in a priori that $\int_{B_2}u=0$ so that $\|u\|_{L^{4}\left(B_{2}\right)}^{2}\lesssim \|\nabla u\|_{L^{4,2}\left(B_{2}\right)}^{2}$.

Finally, combining H\"older's inequality and \eqref{eq: Lp theory of inhomo terms} yields
\[
\int_{B_{\tau}}|u_{2}|^{2}\lesssim
\tau^{8(1-\frac{1}{p})}\left(\int_{B_{1}}|u_{2}|^{\frac{2p}{2-p}}\right)^{\frac{2-p}{p}}
\lesssim\tau^{8(1-\frac{1}{p})}\|f\|_{L^{p}(B_{1})}^{2},
\]
which together with  \eqref{eq: before decay estimate of Delta u}
leads to the decay estimate of $\De u$ for $\tau <1$:
\begin{equation}\label{eq: estimate of Delta u}
\begin{aligned}\int_{B_{\tau}}|\Delta u|^{2} & \lesssim\left(\tau^{4}+\epsilon_{m}^{2}\right)\int_{B_{2}}|\De u|^{2}+\epsilon_{m}^{2}\|\nabla u\|_{L^{4,2}\left(B_{2}\right)}^{2}+
(\tau^{4}+\tau^{8(1-\frac{1}{p})})\|u_{2}\|_{L^{\frac{2p}{2-p}}(B_{1})}^{2}\\
 & \lesssim\left(\tau^{4}+\epsilon_{m}^{2}\right)\int_{B_{2}}|\De u|^{2}+\epsilon_{m}^{2}\|\nabla u\|_{L^{4,2}\left(B_{2}\right)}^{2}
 +\tau^{8(1-\frac{1}{p})}\|u_{2}\|_{L^{\frac{2p}{2-p}}(B_{1})}^{2}\\
 & \lesssim\left(\tau^{4}+\epsilon_{m}^{2}\right)\int_{B_{2}}|\De u|^{2}+\epsilon_{m}^{2}\|\nabla u\|_{L^{4,2}\left(B_{2}\right)}^{2}+\tau^{8(1-\frac{1}{p})}\|f\|_{L^{p}(B_{1})}^{2}.
\end{aligned}
\end{equation}
We used in the second line the fact that $\tau^4\leq \tau^{8(1-\frac{1}{p})}$ since $p<2$ and $\tau<1$.

To continue, we have to estimate the decay of $\|\na u\|_{L^{4,2}(B_r)}$.

2. Decay of $\|\na u\|_{L^{4,2}(B_r)}$.

First use \eqref{eq: Conservation law inhomogeneous again} to rewrite our system as
$$\De \divergence(A\nabla u)=\divergence(\hat{K})+Af\qquad \text{in } B_8,$$
where $\hat{K}=K+\nabla^2A\cdot \nabla u+\nabla A\cdot \nabla^2 u\in L^{\frac{4}{3},1}(B_{10})$ satisfies the similar estimate:
\begin{equation}\label{eq: estimate of K2}
 \|\hat{K}\|_{L^{\frac{4}{3},1}(B_{R})}\lesssim \ep_m(\|\nabla u\|_{L^{4,2}(B_{R})}+\|\nabla^2 u\|_{L^2(B_{R})}).
\end{equation} for any $0<R\le 8$.

Keep in mind that we have extended all the related functions from $B_1$ into $\R^4$ with controlled norms.  By the Hodge decomposition, we have
$$A\D u=\D r+*\D g \qquad \text{in } \R^4$$
where
$$\De^2 r=\De \divergence(A\nabla u)=\divergence(\hat{K})+Af\quad \text{and}\quad \De g=*(\D A\wedge \D u).$$
Denote by $\Ga=c\log(\cdot)$ the fundamental solution of $\De^2$ in $\R^4$, and let $\tilde{r}=\Ga \ast \divergence(\hat{K})$ and $v=\Ga\ast (Ah)$. Then
$\De^2  (r-\tilde{r}-v)=0$ in $B_1$.
Thus there exists a \emph{biharmonic} 1-form $h$ in $B_1$ such that
\[A\D u=\D \tilde{r}+\D v+\ast \D g+h\qquad \text{in } B_1.\]

We estimate the terms above as follows.
Applying the Riesz potential estimates in Proposition \ref{prop:Adams 1975} and the Lorentz-H\"older inequality from Proposition \ref{prop: Lorentz-Holder inequality}, we infer
\begin{eqnarray}\label{eq: decay of nabla u 1}
\begin{aligned}
\|\nabla \tilde{r}\|_{L^{4,2}(\R^4)}&\lesssim \|I_2(\hat{K})\|_{L^{4,2}(\R^4)}\lesssim \|\hat{K}\|_{L^{\frac{4}{3},1}(\R^4)}\lesssim \|\hat{K}\|_{L^{\frac{4}{3},1}(B_1)}\\
&\lesssim \ep_{m}(\|\nabla u\|_{L^{4,2}(B_1)}+\|\nabla^2 u\|_{L^2(B_1)}).
\end{aligned}
\end{eqnarray}
Note that $|\na g|=|\na I_2(\D A\wedge \D u)| \approx |\na^2 I_2 (A\na u)|$. The singular integral theory implies
\begin{equation}\label{eq: decay of nabla u 2}
 \|\na g\|_{L^{4,2}(\R^4)}\lesssim \|A\nabla u\|_{L^{4,2}(\R^4)}
\lesssim \ep_{m}\|\nabla u\|_{L^{4,2}(B_1)}
\end{equation}
Since $\nabla^4\Ga$ is a singular operator,
$$\|\nabla^4v\|_{L^p(\R^4)}\lesssim \|f\|_{L^p(\R^4)}\lesssim \|f\|_{L^p(B_1)}.$$
 Using the Lorentz-Sobolev embedding $ W^{3,p}(\R^4)\subset L^{\bar{p},p}(\R^4)$, where $\bar{p}=\frac{4p}{4-3p}$, we derive
$$\|\nabla v\|_{L^{\bar{p},p}(\R^4)}\lesssim \|f\|_{L^p(\R^4)}\lesssim \|f\|_{L^p(B_1)}.$$
Hence, for any $\ga\in(0,1)$, H\"oder's inequality gives
\[
\|\nabla v\|_{L^{4,2}(B_{\ga})}\lesssim \ga^{4(1-\frac{1}{p})}\|\nabla v\|_{L^{\bar{p},p}(B_{\ga})}\lesssim \ga^{4(1-\frac{1}{p})}\|f\|_{L^p(B_1)}.
\]
%Similarly, in case $p\geq \frac{4}{3}$, we can conclude that
%\[ \|\nabla v\|_{L^{4,2}(B_{\ga})}\lesssim \ga^{\alpha}\|f\|_{L^p(B_1)},\] where $\alpha$ can be any number in $(0,1)$.

%{\color{blue} This only holds for $1<p<4/3$. When  $p\ge4/3$, the scaling gives a number ${1-4/\bar{p}}<1$ since we have to choose $\bar{p}>4$. Also, because when $p> 4/3$, we have ${4(1-\frac{1}{p})}>1$ (so $\ga^{4(1-\frac{1}{p})}<\ga $) and when $p=4/3$, we can not the number $4p/(4-3p)$ in that scaling. But for $p\ge 4/3$, we get an $\alpha$ which could be any number between 0 and 1.}

Now we can conclude  that for any  $\ga\in(0,1)$, there holds
\begin{equation}\label{eq: estimate of the gradient}
\begin{aligned}
\|\nabla u\|_{L^{4,2}(B_{\ga})}  & \le\|h\|_{L^{4,2}(B_{\ga})}+
\|\nabla\tilde{r}\|_{L^{4,2}(B_{\ga})}+\|
\na g\|_{L^{4,2}(B_{\ga})}+\|\nabla v\|_{L^{4,2}(B_{\ga})}\\
& \lesssim \ga\|h\|_{L^{4,2}(B_{1})}+\|\nabla \tilde{r}\|_{L^{4,2}(B_{1})}+\|\na g\|_{L^{4,2}(B_{1})}+\ga^{4(1-\frac{1}{p})}\|f\|_{L^p(B_1)}\\
& \lesssim \ga\|\na u\|_{L^{4,2}(B_{1})}+\ep_m(\|\nabla u\|_{L^{4,2}(B_1)}+\|\nabla^2 u\|_{L^2(B_1)})
 +\ga^{4(1-\frac{1}{p})}\|f\|_{L^p(B_1)}\\
&\le C (\gamma+\ep_m)\Big(\|\nabla u\|_{L^{4,2}(B_2)}+\|\De u\|_{L^{2}(B_2)}\Big)+\ga^{4(1-\frac{1}{p})}\|f\|_{L^p(B_2)},
\end{aligned}
\end{equation}
where we have used \eqref{eq: decay of nabla u 1} and \eqref{eq: decay of nabla u 2} in the last second line, and \eqref{eq: interior L2 estimate} in the last line.

3. Finally, let $\alpha=4(1-1/p)$ and $\be=(\al+1)/2$, and choose $\ep_m\le \ga$ and then choose $\ga\in (0,1)$  such that $2C\ga\le \ga^{\be}$. Combining \eqref{eq: estimate of Delta u} and \eqref{eq: estimate of the gradient} together,  we  infer that   % for $1<p<\frac{4}{3}$
\begin{equation*}\label{eq: decay via CL p < 4/3}
\begin{aligned}
\|\na u\|_{L^{4,2}(B_\ga)}+\|\De u\|_{L^2(B_\ga)}\leq  \ga^\be \left(\|\nabla u\|_{L^{4,2}(B_{2})}+\|\De u\|_{L^2(B_2)}\right)+C\ga^{4(1-\frac{1}{p})}\|f\|_{L^p(B_2)}.
\end{aligned}
\end{equation*}
The proof is complete after a scaling (see Section \ref{sec: preliminaries}) and  an iteration argument as that of Theorem \ref{thm:Sharp-Topping}. We omit the details.
\end{proof}

\begin{remark}\label{rmk:on the decay of Lamm-Riviere in Lp scale}
	Similarly as in the planar case, we may replace all the $L^{4,2}$-norms by the corresponding $L^4$-norms in the above proof to arrive at the following decay estimate in $L^4$-scale,
	\begin{eqnarray}\label{eq: decay Lamm-Riviere via CL in Lp}
	\begin{aligned}
	\|\na u\|_{L^{4}(B_\ga)}+\|\De u\|_{L^2(B_\ga)}\leq  C\ga^{\al}\left(\|\nabla u\|_{L^{4}(B_{1})}+\|\De u\|_{L^2(B_1)}+\|f\|_{L^p(B_1)}\right),
	\end{aligned}
	\end{eqnarray}
	where $\alpha=4\big(1-\frac{1}{p} \big)$ if $1<p<\frac{4}{3}$ and $\alpha$ can be any number in $(0,1)$ if $p\geq \frac{4}{3}$.
	
\end{remark}

The following example (with $n=4$) shows that the H\"older continuity is the best possible regularity that one can expect for the Lamm-Rivi\`ere system \eqref{eq:inhomogenous Lamm-Riviere system} even $f\equiv 0$.

\begin{remark}[A non-Lipschitz continuous example]\label{exam:non-Lipschitz weak solutions}

For any $n\ge 2$, let $B=B_{1/2}(0)\subset\R^{n}$ be the ball centered at the origin with radius
$\frac{1}{2}$ and define $v\colon B\to \R$ as
\[
\ensuremath{v(x)=\left(x_{1}^{2}-x_{2}^{2}\right)(-\log|x|)^{1/2}}.
\]
Direct computation shows
\[
\ensuremath{\Delta v(x)=\frac{x_{2}^{2}-x_{1}^{2}}{2|x|^{2}}\left\{ \frac{n+2}{(-\log|x|)^{1/2}}+\frac{1}{2(-\log|x|)^{3/2}}\right\} }=:f.
\]

Set $V=\big(\frac{f_{x_{1}}}{v_{x_{1}x_{1}}},0,\cdots,0\big)$ and consider $u=v_{x_1}\colon B\to \R$. Then $u\in C^{0,\al}(B)\cap W^{2,2}(B)$
for any $\al\in(0,1)$. It is straightforward to verify that $V\in W^{1,\frac{n}{2}}(B,\R^n)$
and $u$ is a weak solution of
\[
\De u=V\cdot \nabla u\qquad \text{in }B,
\]
which is of the form $\De^2 u=\Delta(V\cdot \nabla u)$. However, note that $u\in C^{\wq}(B\backslash\{0\})$ and
\[
\lim_{x\to0}u_{x_{1}}(x)=\wq.
\]
We thus infer that $u$ is not Lipschitz continuous in $B$.

\end{remark}

\subsection*{Decay estimate for the borderline case}
It is natural to ask whether one can obtain any decay estimate for the borderline case $p=1$. For later use in Theorem \ref{thm:compactness}, we deduce in below a decay estimate for the case $f\in L\log L (B_{10})$.  % This allows us to use the singular integral theory in the borderline case to obtain strong (1,1)-type bound.
Since the proof is rather similar to that used in Theorem \ref{thm:optimal Holder exponent for inho Lamm-Riviere}, we only sketch it for simplicity. It would be interesting to know whether the assumption $f\in L\log L (B_{10})$ can be replaced with  $f\in h^1(B_{10})$, where $h^1$ is the local Hardy space, for definitions see \cite[Appendix A.2]{Sharp-Topping-2013-TAMS}
\begin{proposition}\label{prop: decay estimate for the borderline case}
Let $u\in W^{2,2}(B_{10},\R^m)$ be a weak solution of \eqref{eq:inhomogenous Lamm-Riviere system} and assume  $f\in L\log L (B_{10})$.  Then there exist $0<\ga<1$ and $C>0$ such that  \begin{equation}
\|\De u\|_{L^{2}(B_{\ga})}+\|\nabla u\|_{L^{4,2}\left(B_{\ga}\right)}\le\frac{1}{2}\left(\|\De u\|_{L^{2}(B_{2})}+\|\nabla u\|_{L^{4,2}\left(B_{2}\right)}\right)+C\|f\|_{L^{1}(B_{1})}^{1/2}\|f\|_{L\log L(B_{1})}^{1/2}.\label{eq: decay estimate in LlogL}
\end{equation}
\end{proposition}
\begin{proof}
As in the proof of Theorem \ref{thm:optimal Holder exponent for inho Lamm-Riviere}, we first take $\ep_m$ sufficiently small so that the conservation law holds, and then extend all functions from $B_1$ to $\R^4$ with controlled norms.

1. Decay estimate of $\int_{B_{r}}|\De u|^{2}$.

We shall use the same notations as in Step 1 of the proof of Theorem 1.1. Note that
$f\in L\log L(B_{10})$ implies $Af\in L\log L(B_{10})$ and $\|Af\|_{L\log L(B_{1})}\lesssim\|f\|_{L\log L(B_{1})}$.
So $u_{2}\in W^{2,1}(\R^{4})\subset W^{1,4/3,1}(\R^{4})\subset L^{2,1}(\R^{4})$
and
\[
\|\na^{2}u_{2}\|_{L^{1}(\R^{4})}\lesssim\|Af\|_{L\log L(\R^{4})}\lesssim\|f\|_{L\log L(B_{1})}
\]
and
\[
\|\na^{2}u_{2}\|_{L^{1,\wq}(\R^{4})}\lesssim\|Af\|_{L^{1}(\R^{4})}\lesssim\|f\|_{L^{1}(B_{1})},
\]
from which it follows
\[
\|u_{2}\|_{L^{2,1}(B_{1})}\lesssim\|f\|_{L\log L(B_{1})}
\]
and
\[
\|u_{2}\|_{L^{2,\wq}(B_{1})}\lesssim\|f\|_{L^{1}(B_{1})}.
\]
Therefore, for any $1\le s<\wq$, we have
\[
\|u_{2}\|_{L^{2,s}(B_{1})}\le\|u_{2}\|_{L^{2,1}(B_{1})}^{1/s}\|u_{2}\|_{L^{2,\wq}(B_{1})}^{1-1/s}\lesssim\|f\|_{L^{1}(B_{1})}^{1-1/s}\|f\|_{L\log L(B_{1})}^{1/s}.
\]

Now using (\ref{eq: before decay estimate of Delta u}) and taking
$s=2$ in the above estimate, we obtain
\begin{equation}
\begin{aligned}\int_{B_{\tau}}|\Delta u|^{2} & \lesssim\left(\tau^{4}+\epsilon_{m}^{2}\right)\int_{B_{2}}|\De u|^{2}+\epsilon_{m}^{2}\|\nabla u\|_{L^{4,2}\left(B_{2}\right)}^{2}+\int_{B_{1}}|u_{2}|^{2}\\
 & \lesssim\left(\tau^{4}+\epsilon_{m}^{2}\right)\int_{B_{2}}|\De u|^{2}+\epsilon_{m}^{2}\|\nabla u\|_{L^{4,2}\left(B_{2}\right)}^{2}+\|f\|_{L^{1}(B_{1})}^{1/2}\|f\|_{L\log L(B_{1})}^{1/2}.
\end{aligned}
\label{eq: decay of Delta u in LlogL}
\end{equation}

2. Decay estimate of $\|\nabla u\|_{L^{4,2}\left(B_{r}\right)}^{2}$.

Use the same notations $\tilde{r},r,g,h,v$ as in Step 2 of the
proof of Theorem 1.1. We have $v\in W^{4,1}(\R^{4})$ with
\[
\|\na v\|_{L^{4,1}(\R^{4})}\lesssim\|\na^{4}v\|_{L^{1}(\R^{4})}\lesssim\|f\|_{L\log L(B_{1})}
\]
and
\[
\|\na v\|_{L^{4,\wq}(\R^{4})}\lesssim\|\na^{4}v\|_{L^{1,\wq}(\R^{4})}\lesssim\|f\|_{L^{1}(B_{1})}.
\]
As a result,
\[
\|\na v\|_{L^{4,1}(B_{1})}\lesssim\|f\|_{L\log L(B_{1})}
\]
and
\[
\|\na v\|_{L^{4,\wq}(B_{1})}\lesssim\|f\|_{L^{1}(B_{1})}.
\]
Therefore, for any $1<s<\wq$, we have
\[
\|\na v\|_{L^{4,s}(B_{1})}\le\|\na v\|_{L^{4,1}(B_{1})}^{1/s}\|\na v\|_{L^{4,\wq}(B_{1})}^{1-1/s}\lesssim\|f\|_{L^{1}(B_{1})}^{1-1/s}\|f\|_{L\log L(B_{1})}^{1/s}.
\]
Consequently, we obtain
\begin{equation}
\begin{aligned}\|\nabla u\|_{L^{4,2}\left(B_{\ga}\right)} & \lesssim\left(\ga+\ep_{m}\right)\left(\|\De u\|_{L^{2}(B_{2})}+\|\nabla u\|_{L^{4,2}\left(B_{2}\right)}\right)+\|\na v\|_{L^{4,2}(B_{1})}\\
 & \lesssim\left(\ga+\ep_{m}\right)\left(\|\De u\|_{L^{2}(B_{2})}+\|\nabla u\|_{L^{4,2}\left(B_{2}\right)}\right)+\|f\|_{L^{1}(B_{1})}^{1/2}\|f\|_{L\log L(B_{1})}^{1/2}.
\end{aligned}
\label{eq: Decay of gradient u in LlogL}
\end{equation}

Finally, combining (\ref{eq: decay of Delta u in LlogL}) and (\ref{eq: Decay of gradient u in LlogL}),
we conclude
\[
\|\De u\|_{L^{2}(B_{\ga})}+\|\nabla u\|_{L^{4,2}\left(B_{\ga}\right)}\lesssim\left(\ga+\ep_{m}\right)\left(\|\De u\|_{L^{2}(B_{2})}+\|\nabla u\|_{L^{4,2}\left(B_{2}\right)}\right)+\|f\|_{L^{1}(B_{1})}^{1/2}\|f\|_{L\log L(B_{1})}^{1/2}.
\]
Choosing $\ga,\ep_{m}$ small to obtain the desired estimate.
\end{proof}
\begin{remark} \label{rem: all decay in LlogL} Similarly, one can show that for any $1\le s<\wq$,
\[
\|\De u\|_{L^{2,s}(B_{\ga})}+\|\nabla u\|_{L^{4,s}\left(B_{\ga}\right)}\le\frac 12\left(\|\De u\|_{L^{2,s}(B_{2})}+\|\nabla u\|_{L^{4,s}\left(B_{2}\right)}\right)+C\|f\|_{L^{1}(B_{1})}^{1-1/s}\|f\|_{L\log L(B_{1})}^{1/s}.
\]
\end{remark}

\section{Higher order regularity}\label{sec:higher order Sobolev regularity}

In this section, we shall prove the higher order regularity asserted in Theorem \ref{thm:optimal global estimate for inho Lamm-Riviere}. The key to derive the improved regularity is to use Lemma \ref{lemma:improved Riesz potential}. To illustrate the scheme clearly, we begin with the simple first order case. Throughout this section, we assume $$1<p<4/3$$ and  set
$$\al=4(1-1/p) \quad \text{ and } \quad M\equiv \|u\|_{W^{2,2}(B_1)}+\|f\|_{L^p(B_1)}.$$

\subsection{$W^{1,q}$-estimate with some $q>4$}

The decay estimate in Theorem \ref{thm:optimal Holder exponent for inho Lamm-Riviere} together with H\"older's inequality imply that $\Delta u\in M^{1,2+\alpha}(B_{\frac{1}{2}})$, that is,
\[
\sup_{x\in B_{\frac{1}{2}},0<r<\frac{1}{2}}r^{-(2+\al)}\int_{B_{r}(x)}|\De u|\le CM.
\]

We claim that $\na u\in L^q(B_{1/4})$ with \[\|\na u\|_{L^q(B_{1/4})}\le C_{\al}M, \] where
$$q=\frac{2(2-\alpha)}{1-\alpha}=\frac{2(4-2p)}{4-3p}>4.$$

To prove this claim, take a cut-off function $\eta\in C^{\wq}_0(B_{1/2})$ such that $0\le \eta \le 1$ in $B_{1/2}$, $\eta\equiv 1$ in $B_{1/4}$. Then $\eta u\in W^{2,2}(\R^4)$. Thus
$\eta u=I_2 (-\De (\eta u))$, where $I_2$ is the fundamental solution of $-\De $ in $\R^4$. As a consequence,
\[|\na (\eta u)|\lesssim I_1(\eta |\De u|+|\na \eta||\na u|+|\De \eta ||u|). \]

Easy to verify that $\eta |\De u|\in M^{1,2+\alpha}(\R^4)\cap L^2(\R^4)$. Hence by applying Lemma \ref{lemma:improved Riesz potential} (with $\al=1, \be=2-4(1-1/p), p=1$) we find that  $I_1(\eta |\De u|)\in L^q(\R^4)$ with
\[\|I_1(\eta |\De u|)\|_{L^q(\R^4)}\le  C_{\al}M^{\frac {\al}{2-\al}}\|\eta \De u\|_{L^2(\R^4)}^{1-\frac {\al}{2-\al}}\le C_{\al}M.\]
Note that the lower order term $|\na \eta||\na u|+|\De \eta ||u|$ belongs to $ L^4(\R^4)$. Thus the usual Riesz potential theory implies that
\[\|I_1(|\na \eta||\na u|+|\De \eta ||u|)\|_{L^q(\R^4)}\le C_{\al}M.\]
The claim follows easily from the above two estimates since $\na u=\na (\eta u)$ in $B_{1/4}$.

Note that we have improved the Lebesgue integrability of  $\na u$ from 4 to $q$, even though  $q$ is not the final optimal exponent.

\subsection{$W^{2,q}$-estimate with any $q<\frac{2p}{2-p}$}
We now derive the second order regularity. More precisely, we shall prove the following result.

\begin{proposition}\label{prop:second order Sobolev regularity}
Let $u\in W^{2,2}(B_{10},\R^m)$ be a weak solution of the inhomogenuous system \eqref{eq:inhomogenous Lamm-Riviere system} with $f\in L^p(B_{10})$ for $p\in (1,\frac{4}{3})$.  Then $u\in W^{2,q}_{\loc}(B_{10})$  for any $q<\frac{2p}{2-p}$.
\end{proposition}

\begin{proof}
 By the definition \eqref{eq: K} of $K$ and the decay estimate in Theorem \ref{thm:optimal Holder exponent for inho Lamm-Riviere}, we can easily verify that
\[
\sup_{x\in B_{1/2},0<r<1/2}r^{-\frac{4}{3}\al}\int_{B_{r}(x)}|K|^{4/3}\le CM.
\]
By H\"older's inequality, this implies  that $K\in M^{1,1+\alpha}(B_{1/2})$, that is,
\[
\sup_{x\in B_{1/2},0<r<1/2}r^{-(1+\al)}\int_{B_{r}(x)}|K|\le CM.
\]
Now we extend  $K$ from $B_{1/2}$ into $\R^4$ such that
$$\|K\|_{M^{1,1+\al}(\R^4)}\lesssim \|K\|_{M^{1,1+\al}(B_{1/2})}\lesssim M$$ and
$\|K\|_{L^{4/3}(\R^4)}\lesssim \|K\|_{L^{4/3}(B_{1/2})}$.
Then it follows from Lemma \ref{lemma:improved Riesz potential} (with $\al=1$, $\be=3-\al$, $n=4$, $p=4/3$) that
\[
I_{1}(K)\in L^{\frac{4}{3}\frac{3-\al}{2-\al}}(\R^4).
\]
Write \[q_{0}=\frac{4}{3}\frac{3-\al}{2-\al}.\]
 As a result, $I_{2}({\rm div}K)\approx  I_{1}(K)\in L^{q_{0}}(\R^4)$.
 %A simple computation shows
%\[\frac{3-\al}{2-\al}=\frac{4-p}{4-2p}.\]

%Now we can prove Proposition \ref{prop:second order Sobolev regularity}.
Define $v_1$ and $v_2$ in $\R^4$ as
\begin{eqnarray*}
	v_1=I_{2}({\rm div}K)\approx I_{1}(K), &  & v_2=I_{2}(Af).
\end{eqnarray*}
Here we also extend $A,f$ from $B_{1/2}$ into $\R^4$ with controlled norms.  Then, our previous estimate shows that $v_1\in L^{q_{0}}(\R^4)$ and
\[
v_2\in W^{2,p}(\R^4)\subset L^{2p/(2-p)}(\R^4).
\]
Since $p>1$, we have $q_{0}<\frac{2p}{2-p}$. Thus,  using the fact that $A\De u-v_1-v_2$
is a harmonic function in $B_{1/2}$, we infer that $\De u\in L^{q_0}(B_{1/4}).$ In other words, we obtain
\[u\in W_{\loc}^{2,q_{0}}(B_{1}).\]
Note that $\al>0$ implies $q_0>2$. Thus we have improved the regularity of $u$ from $W^{2,2}$ to $W^{2, q_0}$.
%\textbf{5. $W_{\loc}^{2,q}(B_{1})$ estimate. }
%We need the improved Riesz potential theory again.

Next we use a bootstrapping argument to repeatedly improve the second order regularity of $u$.
We claim that
\begin{eqnarray}\label{eq: bootstrapping W2,q}
\begin{aligned}
u\in W^{2,q}_{loc}\quad \text{with }q<\frac{2p}{2-p} \Longrightarrow u\in W^{2,\frac{4q}{4+q}\frac{3-\alpha}{2-\alpha}}_{loc}.
\end{aligned}
\end{eqnarray}
This is true because if $u\in W^{2,q}_{\loc}$ with $q<\frac{2p}{2-p}$, then the definition \eqref{eq: K} of $K$ implies that $K\in L^{4}\cdot L^{q}\subset L^{\tilde{q}_{0}}$ with $1/\tilde{q}_{0}=1/4+1/q$.
%Note that $\tilde{q}_{0}>4/3$.
%By the previous discussion, we have $\na^{2}u\in L^{q_{0}}$ for some $q_{0}>2$ and $\na u\in L^{s}$. Then, this implies that $K\in L^{4}\cdot L^{q_{0}}\subset L^{\tilde{q}_{0}}$ with $1/\tilde{q}_{0}=1/4+1/q_{0}$. Note that $\tilde{q}_{0}>4/3$, which means that $K$ has better integrability now. So by the improved Riesz potential theory, since $K\in M^{1,1+\la}(B_{1/2})$, we have
Since $K\in M^{1,1+\al}(B_{\frac{1}{2}})$, Lemma \ref{lemma:improved Riesz potential} implies that
\[
v_1\approx I_{1}(K)\in L^{\tilde{q}_{0}\frac{3-\al}{2-\al}}=L^{\frac{4q}{4+q}\frac{3-\alpha}{2-\alpha}}.
\]
Notice that
\[\frac{4q}{4+q}\frac{3-\alpha}{2-\alpha}<\frac{2p}{2-p}\Longleftrightarrow q<\frac{2p}{2-p}\]
and that when $q\nearrow \frac{2p}{2-p}$, we have $\frac{4q}{4+q}\frac{3-\alpha}{2-\alpha}\nearrow \frac{2p}{2-p}$.
Also recall that $v_2\in W^{2,p}\subset L^{\frac {2p}{2-p}}$. Thus the same argument as the above implies that
$\De u\in L^{\frac{4q}{4+q}\frac{3-\alpha}{2-\alpha}}_{\loc}(B_1)$. That is,  $$u\in W^{2,\frac{4q}{4+q}\frac{3-\alpha}{2-\alpha}}.$$  Thus, by iterating the bootstrapping claim \eqref{eq: bootstrapping W2,q}, we find that
$$u\in W^{2,q}_{\loc}\quad \text{for all }q<\frac{2p}{2-p}.$$
The proof of Proposition \ref{prop:second order Sobolev regularity} is complete.
\end{proof}

\begin{remark}\label{rmk:on bootstrap via classical potential theory}
%	We already know that $u\in W^{3,\frac{4}{3}}$ from Lamm-Riviere.
As in the second order case of Sharp-Topping \cite{Sharp-Topping-2013-TAMS}, the classical Calder\'on-Zygmund theory does not give additional improvement on the second order Sobolev exponent. Indeed, by the previous step, we have $\na^{2}u\in L^{q_{0}}$ for some $q_{0}>2$ and $\na u\in L^{s}$. Then, this implies that $K\in L^{4}\cdot L^{q_{0}}\subset L^{\tilde{q}_{0}}$ with $1/\tilde{q}_{0}=1/4+1/q_{0}$. As a result, $v=I_{2}({\rm div}K)\in W^{1,\tilde{q}_{0}}\subset L^{q_{0}}$, and $w=I_{2}(Af)\in W^{2,p}$. Since $A\De u-v-w$ is harmonic, this gives $u\in W^{3,\tilde{q}_{0}}\subset W^{2,q_{0}}$. Note that we do \textbf{not} obtain any improvement for the integrability of $\na^{2}u$. This reflects the importance of Lemma \ref{lemma:improved Riesz potential} in obtaining higher Sobolev regularity.
\end{remark}

%\begin{remark}\label{rmk:compare first order Sobolev exponent}
%	Note that $W^{2,}$
%\end{remark}

\subsection{$W^{3,q}$-estimate with $q>\frac{4}{3}$}
%\textbf{6. $W_{\loc}^{3,\frac{4}{3}+\ep}$ estimate}
With Proposition \ref{prop:second order Sobolev regularity} at hand, we immediately obtain the third order  regularity.
\begin{proposition}\label{prop:third order Sobolev regularity}
	Let $u\in W^{2,2}(B_{10},\R^m)$ be a weak solution of the inhomogenuous system \eqref{eq:inhomogenous Lamm-Riviere system} with $f\in L^p(B_{10})$ for $p\in (1,\frac{4}{3})$.   Then $u\in W^{3,q}_{\loc}(B_1)$  for any $q<\frac{4p}{4-p}$.
\end{proposition}
\begin{proof}
By Proposition \ref{prop:second order Sobolev regularity}, we know that $u\in W_{\loc}^{2,q}$ for all $q<\frac{2p}{2-p}$. As a consequence, the definition \eqref{eq: K} of $K$ implies that  $K\in L^{4}\cdot L^{q}\subset L^{\tilde{q}}$
with $1/\tilde{q}=1/4+1/q$. This implies $v_1:=I_{2}({\rm div}K)\in W^{1,\tilde{q}}$. On the other hand, the classical Calder\'on-Zygmund estimate implies that $v_2=I_{2}(Af)\in W^{2,p}\subset W^{1,\frac{4p}{4-p}}$. Note that
%when $q=\frac{2p}{2-p}$, $\tilde{q}=\frac{4p}{4-p}$ and when
$q<\frac{2p}{2-p}$ if and only if $\tilde{q}<\frac{4p}{4-p}$. Therefore,  by the same argument as that used in the proof of Proposition \ref{prop:second order Sobolev regularity}, we infer that
%\[
%\tilde{q}<\frac{4p}{4-p},(\Leftrightarrow\frac{1}{q}>\frac{1}{p}-\frac{1}{2})
%\]
$u\in W^{3,\tilde{q}}_{\loc}$ for all $\tilde{q}<\frac{4p}{4-p}$.
\end{proof}

\begin{remark}\label{rmk: best 3rd order regularity}
  In the next section, we will show that $u\in W^{2,\frac{2p}{2-p}}_{\loc}$. Then, by the same argument as the above, we conclude that $u\in W^{3,\frac{4p}{4-p}}_{\loc}$.
\end{remark}

We would like to point out that the third order  regularity as obtained in Theorem \ref{thm:optimal Holder exponent for inho Lamm-Riviere} is the best possible, and in general there is no hope to obtain fourth order  regularity.

\begin{example}[Solutions without $W^{4,p}$-regularity]\label{example:no W4p estimate}
Let $g\colon \R\to \R$ be a continuous function with the following properties:
\begin{itemize}
\item $g\in W^{3,2}\big((-1,1)\big)$ but $g\not\in W^{4,1}\big((-1,1)\big)$;
%but $g'''$ fails to be continuous on $(-1,1)$. In particular, $g\notin W^{4,1}\big((-1,1)\big)$;
\item $g\geq 1$ on $(-1,1)$.
\end{itemize}
%For example, one can take $h(t)=t^{\frac{1}{2}}\big(-\log t\big)^{\beta}\chi_{(0,\infty)}(t)$ with $\beta>\frac{1}{2}$ and then set $g$ such that $g'''(t)=h(t)$ for each $t\in \R$ and then integrate back to obtain $g$ with $g(0)=1$. It is easy to verify that $g$ satisfies the above properties.

Consider the map $u\colon B_1\to \R$, $B_1\subset \R^4$, defined by
$$u(x)=x_1g(x_2).$$
Set
$$V_1(x)=x_1\frac{g''(x_2)}{g(x_2)}\quad \text{and}\quad V(x)=(V_1(x),0,0,0).$$
It is straightforward to verify that $V\in W^{1,2}\big(B_1\big)$ and
\begin{equation*}
\Delta^2 u=\Delta\big(V\cdot \nabla u\big)\quad \text{in } B_1.
\end{equation*}
However, the regularity of $g$ implies that $u\notin W^{4,1}(B_1)$.

\end{example}
%{\color{blue}Pay attention to domains, loc etc in the final check!}

\section{Optimal local estimates}\label{sec:optimal global estimates}

In this section we complete the proof of Theorem \ref{thm:optimal global estimate for inho Lamm-Riviere}.
First we give the following lemma for later usage.
\begin{lemma} There exists a constant $C=C(p)>0$ satisfying the
following property. Let $B_{1}=B_{1}(0)\subset\R^{4}$. For any $x_{0}\in B_{1/2}(0)$
and $0<R<1/2$, if $h$ satisfies the equation
\begin{equation}
\begin{cases}
\De h=0 & \text{in }B_{R}(x_{0}),\\
h=A\De u & \text{on }\pa B_{R}(x_{0}),
\end{cases}\label{eq: harmonic part}
\end{equation}
then
\begin{equation}\label{eq: uniform estimate for harmonic part}
\|h\|_{L^{\bar{p}}(B_{R/2}(x_{0}))}\le C\left(\|u\|_{W^{2,2}(B_{1})}+\|f\|_{L^{p}(B_{1})}\right).
\end{equation}
\end{lemma}
\begin{proof}
Since $u\in W^{3,4/3}(B_{1})$, the existence of $h$ for equation
(\ref{eq: harmonic part}) can be easily deduced from Lemma \ref{lem: a W1p regularty lemma}.
Take a scaling transform $u_{R}(x)=u(x_{0}+Rx)$, $A_{R}=A(x_{0}+Rx)$
and $h_{R}(x)=R^{2}h(x_{0}+Rx)$ for $x\in B_{1}$ such that
\[
\begin{cases}
\De h_{R}=0 & \text{in }B_{1},\\
h_{R}=A_{R}\De u_{R} & \text{on }\pa B_{1}.
\end{cases}
\]
Applying Lemma \ref{lem: a W1p regularty lemma} (with $n=4$, $p=4/3$),
we have $h_{R}\in W^{1,4/3}(B_{1})$ and
\[\begin{aligned}
\|\na h_{R}\|_{L^{4/3}(B_{1})}&\le\|\na(A_{R}\De u_{R})\|_{L^{4/3}(B_{1})}\\
&\le C\left(\|\na\De u_{R}\|_{L^{4/3}(B_{1})}+\|\na A_{R}\|_{L^{4}(B_{1})}\|\De u_{R}\|_{L^{2}(B_{1})}\right)
\end{aligned}
\]
for some constant $C>0$ by \eqref{eq: estimate of gradient}. Since $\na A\in L^{4}(B_{1})$,
$\|\na A_{R}\|_{L^{4}(B_{1})}=\|\na A\|_{L^{4}(B_{R}(x_{0}))}$ is
uniformly bounded with respect to $x_{0}$ and $R$. Hence,
\[
\|\na h_{R}\|_{L^{4/3}(B_{1})}\le C\left(\|\na\De u\|_{L^{4/3}(B_{R}(x_{0}))}+\|\De u_{R}\|_{L^{2}(B_{R}(x_{0}))}\right).
\]
Then, applying Theorem 1.1, we obtain
\[
\|\na h_{R}\|_{L^{4/3}(B_{1})}\le C(p)\left(\|u\|_{W^{2,2}(B_{1})}+\|f\|_{L^{p}(B_{1})}\right)R^{4(1-1/p)}.
\]
Here we used a simple fact that $$\|\na\De u\|_{L^{4/3}(B_{R}(x_{0}))}\le C(p)\left(\|u\|_{W^{2,2}(B_{1})}+\|f\|_{L^{p}(B_{1})}\right)R^{4(1-1/p)}.$$ We leave the  proof  for interested readers.  As a consequence,
\[
\begin{aligned}\|h_{R}\|_{L^{2}(B_{1})} & \le\|h_{R}-A_{R}\De u_{R}\|_{L^{2}(B_{1})}+\|A_{R}\De u_{R}\|_{L^{2}(B_{1})}\\
 & \le C_{p}\left(\|u\|_{W^{2,2}(B_{1})}+\|f\|_{L^{p}(B_{1})}\right)R^{4(1-1/p)}.
\end{aligned}
\]
In particular, this implies that
\[
\|h_{R}\|_{L^{\bar{p}}(B_{1/2})}\le C\|h_{R}\|_{L^{2}(B_{1})}\le C_{p}\left(\|u\|_{W^{2,2}(B_{1})}+\|f\|_{L^{p}(B_{1})}\right)R^{4(1-1/p)},
\]
which is equivalent to \eqref{eq: uniform estimate for harmonic part}.
The proof is complete.
\end{proof}

%We start from the optimal third-order  regularity.

%\begin{proposition}\label{prop:optimal w3,q estimate}
%	%There exist $\ep=\ep(p,m)>0$ and $C=C(p,m)<\infty$ such that if ..., then
%	Under the assumptions of Theorem \ref{thm:optimal global estimate for inho Lamm-Riviere}, we have
%	\[
%	\|u\|_{W^{3,\frac{4p}{4-p}}(B_{\frac{1}{2}})}\le C\left(\|f\|_{L^{p}(B_{1})}+\|u\|_{L^{1}(B_{1})}\right).
%	\]
%\end{proposition}

Now we can prove Theorem \ref{thm:optimal global estimate for inho Lamm-Riviere}.
\begin{proof}[Proof of Theorem \ref{thm:optimal global estimate for inho Lamm-Riviere}]
Set $\bar{q}=\frac{4p}{4-p}$. The idea is to establish a uniform estimate for $\|\nabla^3 u\|_{L^q(B_{{1}/{2}})}$ in terms of $\left(\|f\|_{L^{p}(B_{1})}+\|u\|_{L^{1}(B_{1})}\right)$. The proof consists of two steps. In the first step, we  prove $u\in W^{3,\bar{q}}_{\loc}$. By Remark \ref{rmk: best 3rd order regularity}, it suffices to show $u\in W^{2,\bar{p}}_{\loc}$ for $\bar{p}=\frac{2p}{2-p}$. In the second step we deduce the desired estimate.

We have proved that $u\in W_{\loc}^{2,\ga}(B_{1})$ for any $\ga<{2p}/(2-p)$
whenever $1<p<4/3$.  The idea is to show that $\|\na^{2}u\|_{L^{\gamma}(B_{{1}/{4}})}$ is uniformly bounded from above with respect to $\ga<{2p}/(2-p)$.

In the below, let $$\bar{p}=\frac{2p}{2-p}\quad\text{ and } \gamma\in (\frac{\bar{p}}{2},\bar{p}).$$ By  \eqref{eq: B2}, there exists a constant $C>0$  depending only on $p$,  such that
\begin{equation}\label{eq: L2 theory}
	\|\na^{2}u\|_{L^{\gamma}(B_{{1}/{4}})}\le C\left(\|\De u\|_{L^{\gamma}(B_{{1}/{2}})}+\|u\|_{L^{1}(B_{{1}/{2}})}\right)
\end{equation}
holds for all $\gamma\in (\frac{\bar{p}}{2},\bar{p})$.

 Decompose $A\Delta u=v+h$ in $B_1$ such that $h$ is a  harmonic function in $B_{{1}}$ and  $v=0$ on $\pa B_1$. From \eqref{eq:conservation law of Lamm Riviere} we have
 % the following conservation law \begin{equation*} 	\Delta(A\De u)=\divergence(K)+Af, \end{equation*} where
 $$\Delta v=\divergence(K)+Af$$  in $B_{1}$,  with $K$ being given by $$K=2\nabla A\cdot \De u-\De A\nabla u+Aw\nabla u-\nabla A(V\cdot \nabla u)+A\nabla\big(V\cdot \nabla u \big)+B\cdot \nabla u.$$

We first estimate $\|v\|_{L^{\gamma}(B_{{1}})}$. To this end, notice by duality that
$$\|v\|_{L^\gamma(B_{{1}})}=\sup_{\varphi\in C^{\wq}_0(B_{{1}}),\ \|\varphi\|_{L^{\gamma'}(B_{{1}})}\leq 1}\int_{B_{{1}}}v\varphi dx,$$
where $\gamma'=\frac{\gamma}{\gamma-1}$ is the conjugate exponent of $\gamma$. Let $\psi$ be the solution to the Dirichlet problem $\De \psi=\varphi$ on $B_{{1}}$ with $\psi=0$ on $\partial B_{{1}}$. Since $\bar{p}/2<\ga<\bar{p}$, we have
$$\frac{3}{2}+\frac{4}{3(3p-2)}=\bar{p}'<\gamma'<\big(\frac{\bar{p}}{2}\big)'=\frac{1}{2}+\frac{1}{p-1}.$$
Combining this bound together with the Calder\'on-Zygmund theory (see Section \ref{sec:C-Z theorey revisited}), there exists a constant $C=C(p)>0$ independent of $\ga$ such that
$$\|\psi\|_{W^{2,\gamma'}(B_{{1}})}\leq C \|\varphi \|_{L^{\gamma'}(B_{{1}})}\le C.$$
Let $\gamma^\ast=(\gamma')^\ast$ be defined by
$$\frac{1}{\gamma^\ast}=\frac{1}{\gamma'}-\frac{1}{4}.$$
Then we infer that
\[
\begin{aligned}
	 \|v\|_{L^{\gamma}(B_{{1}})}\le  C_p  \sup_{\psi\in W^{2,\gamma'}\cap W_0^{1,\gamma^\ast}(B_{{1}}),\ \|\psi\|_{W^{2,\gamma'}(B_{{1}})}\leq 1}\int_{B_{{1}}}v\De \psi dx.	
\end{aligned}
\]
Now we estimate the above supremum as follows. Since $v=0, \psi =0 $ on $\pa B_1$, we have
\begin{equation}\label{eq: 6.2}
	\int_{B_{{1}/{2}}}v\De \psi dx=\int_{B_{{1}/{2}}}\De v \psi dx=
	\int_{B_{{1}/{2}}}K\cdot \na \psi+Af\psi dx.
\end{equation}
Note that $\ga'>{\bar{p}'}$ since $\ga<\bar{p}$, where $\bar{p}'$ is the conjugate exponent of $\bar{p}$.  Thus, by the Sobolev embedding $W^{2,\bar{p}'}(B_{{1}})\subset L^{p'}(B_{1})$ and H\"older's inequality, we get
$$\|\psi\|_{L^{p'}(B_{{1}})}\le C_p \|\psi\|_{W^{2,\bar{p}'}(B_{{1}})}\le C_p \|\psi\|_{W^{2,\gamma'}(B_{{1}})}\leq C_p.$$
Hence
\begin{equation}\label{eq: 6.3}
	\int_{B_{{1}}}Af\psi dx\lesssim \|f\|_{L^{p}(B_{{1}})}\|\psi\|_{L^{p'}(B_{{1}}) }\le C_p\|f\|_{L^{p}(B_{{1}})}.
\end{equation}
For the  integral $\int_{B_{{1}}}K\cdot \na \psi dx$, we estimate term by term by  H\"oler's inequality and the smallness assumption. For the first term $\na A\De u$ of $K$, we have
\[\int_{B_{1}}|\na A\De u\na \psi |\le \|\na A\|_{L^4(B_{1})}\|\De u\|_{L^{\ga}(B_{1})} \|\psi\|_{L^{\ga^{\ast}}(B_{1})}\lesssim \ep_m \|\De u\|_{L^{\ga}(B_{1})}.\]
The rest terms can be  estimated similarly. This finally leads us to
\begin{equation}\label{eq: 6.4}
	\int_{B_{{1}}}K\cdot \nabla \psi dx\lesssim\ep_m\left(\|\na^{2}u\|_{L^{\gamma}(B_{{1}})}+\|\na u\|_{L^{\gamma^\ast}(B_{{1}})}\right).
\end{equation}
Therefore, for any $\psi\in W^{2,\gamma'}\cap W_0^{1,\gamma^\ast}(B_{{1}})$ with $\|\psi\|_{W^{2,\gamma'}(B_{{1}})}\leq 1$, \eqref{eq: 6.2}-\eqref{eq: 6.4} implies
\begin{equation}\label{eq: 6.5}
	\int_{B_{{1}}}v\De \psi dx \le C_p \ep_m\left(\|\na^{2}u\|_{L^{\gamma}(B_{{1}})}+\|\na u\|_{L^{\gamma^\ast}(B_{{1}})}\right) + C_p\|f\|_{L^{p}(B_{{1}})}.
\end{equation}

%and
%$$\int_{B_{\frac{1}{2}}}Af\psi dx\lesssim \|f\|_{L^{p}(B_{\frac{1}{2}})} {\color{blue} \text{ (it seems we can even take } p=1) }.$$
%{\color{red} Recomputed in the above.}

It remains to estimate $\|\na u\|_{L^{\gamma^\ast}(B_{{1}})}$. Using the Sobolev embedding theorem, similar to the estimate (30) of \cite{Sharp-Topping-2013-TAMS}, we may find a constant $C>0$, independent of $t$, such that for any $t\in(1,4)$,
\[
\|\na u\|_{L^{\frac{4t}{4-t}}(B_{{1}})}\le\frac{C}{4-t}\|u\|_{W^{2,t}(B_{{1}})}
\le\frac{C}{4-t}\left(\|\na^{2}u\|_{L^{t}(B_{{1}})}
+\|u\|_{L^1(B_{{1}})}\right).
\]
Applying this estimate with $t=\gamma$, and taking supremum  with respect to $\psi$ in  \eqref{eq: 6.5}, we achieve
\begin{equation*}
	\|v\|_{L^\gamma(B_{{1}})} \le C_p \ep_m\left(\|\na^{2}u\|_{L^{\gamma}(B_{{1}})}+\| u\|_{L^1(B_{{1}})}\right) + C_p\|f\|_{L^{p}(B_{{1}})}.
\end{equation*}
Thus
\begin{equation}\label{eq: 6.6}
\begin{aligned}
	\|A\Delta u\|_{L^{\ga}(B_{\frac{1}{2}})}&\leq \|v\|_{L^{\ga}(B_{\frac{1}{2}})}+\|h\|_{L^\ga(B_{\frac{1}{2}})}\\
	&\leq C_p \ep_m\left(\|\na^{2}u\|_{L^{\gamma}(B_{{1}})}+\| u\|_{L^1(B_{{1}})}\right) + C_p\left(\|f\|_{L^{p}(B_{{1}})}+\|h\|_{L^{\bar{p}}(B_{\frac{1}{2}})}\right).
\end{aligned}
\end{equation}
Finally, we infer from  \eqref{eq: L2 theory}  and  \eqref{eq: 6.6} that
\begin{equation}\label{eq:key second order estimate}
	\|\na^{2}u\|_{L^\gamma(B_{\frac{1}{4}})}\le C\ep_m\|\na^{2}u\|_{L^\gamma(B_{{1}})}
	+C(\|f\|_{L^p(B_{{1}})}+\|u\|_{L^1(B_{{1}})}+\|h\|_{L^{\bar{p}}(B_{\frac{1}{2}})})
\end{equation} holds for  some $C=C(p,m)>0$ which is independent of $\ga$.

With \eqref{eq:key second order estimate} at hand, the remaining step is to use a standard scaling technique as that of  \cite[Proof of Lemma 7.2]{Sharp-Topping-2013-TAMS}. Namely, we first use scaling to deduce, for any $B_{R}(z)\subset B_{1}$,
\[
\|\na^{2}u\|_{L^\gamma(B_{\frac{R}{4}}(z))}\le C\ep\|\na^{2}u\|_{L^\ga(B_{R}(z))}+CR^{-6}(\|f\|_{L^p(B_{R}(z))}
+\|u\|_{L^1(B_{R}(z))}+\|h\|_{L^{\bar{p}}(B_{\frac{R}{2}})}),
\]
for $\be=6\bar{p}>0$ (independent of $\ga$). At this moment, \eqref{eq: uniform estimate for harmonic part} implies that we have
\[
\|\na^{2}u\|_{L^\gamma(B_{\frac{R}{4}}(z))}\le C\left(\|u\|_{W^{2,2}(B_{1})}+\|f\|_{L^{p}(B_{1})}\right)
\]
for all $B_{R}(z)\subset B_{1}$. Then, we use an iteration lemma of Simon (see e.g. \cite[Lemma A.7]{Sharp-Topping-2013-TAMS}) to derive the uniform estimate with respect to $\ga$:
\begin{equation}\label{eq: uniform estimate step 1}
\|\na^{2}u\|_{L^\ga(B_{{1}/{4}})}\le C\left(\|u\|_{W^{2,2}(B_{1})}+\|f\|_{L^{p}(B_{1})}\right)
\end{equation}
with a constant $C=C(p,m)$ independent of $\ga$. Letting $\ga\to\bar{p}$
yields $\na^{2}u\in L^{\bar{p}}(B_{{1}/{4}})$. Consequently, $u\in W^{3,\bar{q}}(B_{\frac{1}{4}})$.

In the second step, we want to refine estimate \eqref{eq: uniform estimate step 1} to obtain  the following quantitative estimate:
\begin{equation}\label{eq:3}
\|u\|_{W^{3,\bar{q}}(B_{\frac{1}{2}})}\le C\left(\|f\|_{L^{p}(B_{1})}+\|u\|_{L^{1}(B_{1})}\right).
\end{equation} We use  an interpolation argument.

By the conservation law, we have
$$\divergence\left(\nabla(A\Delta u) \right)=\divergence(K)+Af.$$
Thus  elliptic regularity theory implies that
$$\|\nabla(A\Delta u)\|_{L^{\bar{q}}(B_{\frac{1}{2}})}\lesssim \|K\|_{L^{\bar{q}}(B_1)}+\|f\|_{L^{p}(B_1)}+\|A\Delta u\|_{L^{\bar{p}}(B_1)},$$
from which it follows
\[
	\|\nabla \Delta u\|_{L^{\bar{q}}(B_{\frac{1}{2}})}\lesssim \ep_{m}\left(\|\Delta u\|_{L^{\bar{p}}(B_1)}+\|\nabla u\|_{L^{p_1}(B_1)}\right)+\|\Delta u\|_{L^{\bar{q}}(B_1)}+\|f\|_{L^p(B_1)},
\] where $p_1=4p/(4-3p)$ is the Sobolev exponent of $W^{2,\bar{q}}(\R^4)$ embedding into $L^{p_1}(\R^4)$.
On the other hand, by the Sobolev embedding and interpolation inequality, there holds
$$\|\Delta u\|_{L^{\bar{p}}(B_1)}\lesssim \|\nabla\Delta u\|_{L^{\bar{q}}(B_1)}+\|\Delta u\|_{L^{\bar{p}}(B_1)}\lesssim \|\nabla\Delta u\|_{L^{\bar{q}}(B_1)}+\|u\|_{L^{1}(B_1)}$$
and
$$\|\Delta u\|_{L^{\bar{q}}}\leq \epsilon \|\nabla\Delta u\|_{L^{\bar{q}}(B_1)}+C_{\epsilon}\|u\|_{L^1(B_1)}.$$
Combining all these estimates, we thus conclude that
\begin{equation}\label{eq:4}
\|\nabla\Delta u\|_{L^{\bar{q}}(B_{\frac{1}{2}})}\lesssim \epsilon \|\nabla\Delta u\|_{L^{\bar{q}}(B_1)}+C\left(\|f\|_{L^{p}(B_{1})}+\|u\|_{L^{1}(B_{1})} \right).
\end{equation}
From \eqref{eq:4}, a  scaling argument as that of \eqref{eq: uniform estimate step 1} gives \eqref{eq:3}. The proof is complete.
\end{proof}

Now we can prove Corollary \ref{coro:energy gap}.

\begin{proof}[Proof of Corollary \ref{coro:energy gap}]
	By the proof of previous proposition, we know there exists a constant $C=C(p,m)>0$ such that
	$$\|\nabla^2u\|_{L^{\bar{p}}(B_{1/2})}+\|\nabla u\|_{L^{\frac{4p}{4-3p}}(B_{1/2})}\leq C(p,m)\|u\|_{L^1(B_1)}.$$
	 Using a simple scaling, we then deduce
	\begin{equation}\label{eq:for engergy gap}
		\|\nabla^2 u\|_{L^{\bar{p}}(B_R)}+\|\nabla u\|_{L^{\frac{4p}{4-3p}}(B_R)}\leq C(p,m)R^{-4(1-1/p)}\|u\|_{W^{2,2}(\R^2)}.
	\end{equation}
Sending $R\to \infty$ gives $\nabla u=0$ in $\R^4$,  and so $u$ is a constant. Since $u\in L^2(\R^4)$, $u\equiv 0$ in $\R^4$.
\end{proof}

\subsection{Optimal $W^{4,p}$ estimate in a special case}
In this section, we deduce $W^{4,p}$ estimate under the additional assumption  $V\in W^{2,\frac 43}$ and $w\in W^{1, \frac 43}$. Note that the system of biharmonic mappings is included in this case.

\begin{proposition}\label{prop:optimal w4,p estimate}
	%There exist $\ep=\ep(p,m)>0$ and $C=C(p,m)<\infty$ such that if ..., then
	Under the assumptions of Theorem \ref{thm:optimal global estimate for inho Lamm-Riviere}, if in addition $V\in W^{2,\frac 43}(B_{10})$ and $w\in W^{1, \frac 43}(B_{10})$, then $u\in W^{4,p}_{\loc}(B_1)$ and
\begin{equation}\label{eq: optimal W 4p estimate}
\|u\|_{W^{4,p}(B_{\frac{1}{2}})}\le C\left(\|f\|_{L^{p}(B_{1})}+\|u\|_{L^{1}(B_{1})}\right).
\end{equation}
\end{proposition}

\begin{proof}
By the conservation law, we have
\[
\De^{2}u+2A^{-1}\na A\cdot\na\De u+A^{-1}\De A\De u=A^{-1}{\rm div}K+f.
\]
Equivalently, we have
\begin{equation}\label{eq: equivalent Conservationi law}
	\De^{2}u={\rm div}(A^{-1}K)+\tilde{f},
\end{equation}
where $\tilde{f}=f+\na A^{-1}\cdot K-2A^{-1}\na A\cdot\na\De u-A^{-1}\De A\De u\in L^{p}(B_{2})$
with
\[
\|\tilde{f}\|_{L^{p}(B_{1})}\lesssim\|f\|_{L^{p}(B_{1})}+\ep_{m}\left(\|\na^{3}u\|_{L^{p^{*}}(B_{1})}+\|\na^{2}u\|_{L^{\bar{p}}(B_{1})}+\|\na u\|_{L^{\frac{4p}{4-3p}}(B_{1})}\right).
\]	
	
We use \eqref{eq: equivalent Conservationi law} to prove the desired result.
%Let $\tilde{f}$ be defined as in the proof of Theorem \ref{thm:optimal global estimate for inho Lamm-Riviere}.
Since we have proved  $u\in W^{3,\frac{4p}{4-p}}$, under the additional assumptions $V\in W^{2,\frac 43}(B_{10}), w\in W^{1, \frac 43}(B_{10})$, it is straightforward to verify  that
$
K\in W^{1,p}_{\loc}(B_1),
$
with \[
\|\na K\|_{L^{p}(B_{\frac{2}{3}})}\lesssim
\|u\|_{W^{3,\frac{4p}{4-p}}(B_{\frac{3}{4}})}
\lesssim(\|f\|_{L^{p}(B_{1})}+\|u\|_{L^{1}(B_{1})}).
\]
As a consequence, both $\tilde{f}$ and ${\rm div} A^{-1}K$ belong to $ L^p_{\loc}(B_1)$ with the estimate
\[
\|{\rm div} (A^{-1}K)\|_{L^{p}(B_{\frac{2}{3}})}+\|\tilde{f}\|_{L^{p}(B_{\frac{2}{3}})}\lesssim\ep
\|u\|_{W^{3,\frac{4p}{4-p}}(B_{\frac{3}{4}})}
\lesssim\ep(\|f\|_{L^{p}(B_{1})}+\|u\|_{L^{1}(B_{1})}).
\]
% $v_1=I_{2}({\rm div}K)\in W^{2,p}(B_{\frac{1}{2}})$ with
%\[ \|\na^{2}v_1\|_{L^{p}(B_{\frac{1}{2}})}\lesssim\|{\rm div}K\|_{L^{p}(B_{\frac{2}{3}})}\lesssim\ep(\|f\|_{L^{p}(B_{1})}+\|u\|_{L^{1}(B_{1})}).\]
%Since $v_2=I_{2}(Af)\in W^{2,p}(B_{1})$ as well, we deduce that $A\De u=v_1+v_2+h$ belongs to $W_{\loc}^{2,p}(B_{1/2})$, from which we find that $u\in W_{\loc}^{4,p}(B_{1})$ Here we used the fact that $A^{-1}\in W^{2,2}$.
% Note that
%\[
%\begin{aligned}
%\|\na^{4}u\|_{L^{p}(B_{\frac{1}{4}})}&\lesssim \|\nabla^2 \De u\|_{L^{p}(B_{\frac{1}{2}})}+\|u \|_{L^{1}(B_{\frac{1}{2}})}\\
%&\lesssim \|\nabla^2 (A\De u)\|_{L^{p}(B_{\frac{1}{2}})}+\|u \|_{L^{1}(B_{\frac{1}{2}})}\\
%&\lesssim \|\nabla^2 I_2(\divergence(K))\|_{L^{p}(B_{\frac{1}{2}})}+\|\nabla^2I_2(Af)\|_{L^{p}(B_{\frac{1}{2}})} +\|u \|_{L^{1}(B_{\frac{1}{2}})} \\
%&\lesssim \|\nabla K\|_{L^{p}(B_{\frac{1}{2}})}+\|f\|_{L^{p}(B_{\frac{1}{2}})}+\|u \|_{L^{1}(B_{\frac{1}{2}})} \\
%&\lesssim \|f\|_{L^{p}(B_{1})}+\|u\|_{L^{1}(B_{1})}.
%\end{aligned} \]
Hence, combining \eqref{eq: equivalent Conservationi law} with the standard elliptic regularity theory, we achieve  the desired estimate \eqref{eq: optimal W 4p estimate}. The proof is complete.
\end{proof}

\section{Borderline case and the compactness result}\label{sec:compactness}
In this section we prove Theorem \ref{thm:optimal local estimate LlogL} and Theorem \ref{thm:compactness}.

\begin{proof}[Proof of Theorem \ref{thm:optimal local estimate LlogL}]
The argument is quite similar as that used in Theorem \ref{thm:optimal global estimate for inho Lamm-Riviere}, so we only sketch the proof.

We shall use the conservation law \eqref{eq:conservation law of Lamm Riviere} as there. First extend all relevant functions from $B_1$ to $\R^4$ with controlled norms.
%\begin{equation*}%\label{eq:equivalent system via conservation law} \Delta(A\De u)=\divergence(K)+Af, \end{equation*}
%where
%$$K=2\nabla A\cdot \De u-\De A\nabla u+Aw\nabla u-\nabla A(V\cdot \nabla u)+A\nabla\big(V\cdot \nabla u \big)+B\cdot \nabla u.$$
Then let $v_1=I_2\big(\divergence(K) \big)$, $v_2=I_2(Af)$ and $h:=A\Delta u-v_1-v_2$. As before, $h$ is a harmonic function in $B_1$. The condition $K\in L^{\frac{4}{3},1}$ implies  $v_1\in W^{1,\frac{4}{3},1}(B_{1})$ with the estimate
\[\|\na v_1\|_{L^{4/3,1}(B_1)}\lesssim \|K\|_{L^{4/3,1}(B_1)}\lesssim \ep_m (\|\na^2 u\|_{{L^{2,1}(B_1)}}+\|\na u\|_{L^{4,1}(B_1)}).\] Since $f\in L\log L(B_1)$, we have $Af\in L\log L(B_1)\subset h^1(\R^4)$, where $h^1(\R^4)$ is again the local Hardy space (see \cite[Appendix A.2]{Sharp-Topping-2013-TAMS}). Then the singular integral theory implies that $v_2\in W^{2,1}(B_{1})\subset W^{1,\frac{4}{3},1}(B_1)$ together with the estimate
\[\|\na^2 v_2\|_{L^{1}(B_1)}+\|\na v_2\|_{L^{4/3,1}(B_1)}\lesssim \|f\|_{L\log L(B_1)}.\]
Hence $A\De u=v_1+v_2+h\in W^{1,\frac{4}{3},1}(B_{\frac{7}{8}})$. In particular, this implies that $u\in W^{3,\frac{4}{3},1}(B_{\frac{7}{8}})$.

Next using the same arguments as in the proof of Proposition \ref{prop:optimal w2,q estimate}, we obtain
\[\|\De u\|_{L^{2,1}(B_{7/8})}+\|\na u\|_{L^{4,1}(B_{7/8})}\lesssim  \|f\|_{L\log L(B_1)} + \|u\|_{L^1(B_1)}.\]
Consequently,
\[
\|K\|_{L^{4/3,1}(B_{7/8})}\lesssim\|\na^{2}u\|_{L^{2,1}(B_{7/8})}+\|\na u\|_{L^{4/3,1}(B_{7/8})}\lesssim\|f\|_{L^{4/3,1}(B_{1})}+\|u\|_{L^{1}(B_{1})}.
\]

Returning to system \eqref{eq:conservation law of Lamm Riviere},  the elliptic regularity theory yields
\[
\begin{aligned}\|\na(A\De u)\|_{L^{4/3,1}(B_{3/4})} & \lesssim\|K\|_{L^{4/3,1}(B_{7/8})}+\|Af\|_{L\log L(B_{7/8})}+\|A\De u\|_{L^{2,1}(B_{7/8})}\\
 & \lesssim\|f\|_{L^{4/3,1}(B_{1})}+\|u\|_{L^{1}(B_{1})}.
\end{aligned}
\]
Hence, combining the interior $L^2$-theory and the above estimates, we obtain
\[
\begin{aligned}\|\na^{3}u\|_{L^{4/3,1}(B_{1/2})} & \lesssim\|\De\na u\|_{L^{4/3,1}(B_{3/4})}+\|\na u\|_{L^{4,1}(B_{3/4})}\\
 & \lesssim\|\na(A\De u)\|_{L^{4/3,1}(B_{3/4})}+\|\na u\|_{L^{4,1}(B_{3/4})}\\
 & \lesssim\|f\|_{L^{4/3,1}(B_{1})}+\|u\|_{L^{1}(B_{1})}.
\end{aligned}
\]
 The proof is complete. \end{proof}

% Using the standard interpolation theory (see e.g.~\cite[Chapter 5]{Adams-book}), we may estimate as follows:
%\[ \begin{aligned}
%\|\na^{3}u\|_{L^{\frac{4}{3},1}(B_{\frac{1}{2}})}&\lesssim \|\nabla \De u\|_{L^{\frac{4}{3},1}(B_{\frac{3}{4}})}+\|u \|_{L^{1}(B_{\frac{3}{4}})}\\
%&\lesssim \|\nabla (A\De u)\|_{L^{\frac{4}{3},1}(B_{\frac{3}{4}})}+\|u \|_{L^{1}(B_{\frac{3}{4}})}\\
%&\lesssim \|\nabla I_2(\divergence(K))\|_{L^{\frac{4}{3},1}(B_{\frac{3}{4}})}+\|\nabla I_2(Af)\|_{L^{\frac{4}{3},1}(B_{\frac{3}{4}})}+\|u \|_{L^{1}(B_{\frac{3}{4}})} \\
%&\lesssim \|K\|_{L^{\frac{4}{3}}(B_{\frac{3}{4}})}+\|f\|_{L\log L(B_{\frac{3}{4}})}+\|u \|_{L^{1}(B_{\frac{3}{4}})} \\
%&\lesssim \epsilon_m\big(\|\nabla u\|_{L^{4,2}(B_{\frac{3}{4}})}+\|\nabla^2 u\|_{L^{2}(B_{\frac{3}{4}})}\big)+\|f\|_{L\log L(B_{\frac{3}{4}})}+\|u \|_{L^{1}(B_{\frac{3}{4}})}\\
%&\lesssim \epsilon_m\delta \|\na^{3}u\|_{L^{\frac{4}{3},1}(B_{1})}+(C(\delta)\epsilon_m+1)\|u\|_{L^1(B_1)}
%\|f\|_{L\log L(B_{1})}.\end{aligned}\]
%By a standard scaling, we deduce, for any $B_R(z)\subset B_1$, there holds
%\[\|\nabla^3 u\|_{L^{\frac{4}{3},1}(B_{\frac{R}{2}}(z))}\leq \|\nabla^3 u\|_{L^{\frac{4}{3},1}(B_{{R}}(z))}+R^{-4}\big(\|u\|_{L^1(B_R(z))}+\|f\|_{L\log L(B_R(z))}\big).\]
%Using an iteration lemma of Simon again, we find that
%\[\|\nabla^3 u\|_{L^{\frac{4}{3},1}(B_{\frac{1}{2}})}\leq C\big(\|u\|_{L^1(B_1)}+\|f\|_{L\log L(B_1)}\big).\]

Next we follow the idea of Sharp and Topping \cite{Sharp-Topping-2013-TAMS} to apply Theorem \ref{thm:optimal local estimate LlogL} to prove  Theorem \ref{thm:compactness}.
\begin{proof}
Fix a  ball $B_R(x)\subset\subset B_1$. By Theorem \ref{thm:optimal local estimate LlogL}, we know $\{u_n\}$ is uniformly bounded in $W^{3,\frac{4}{3},1}(B_R)$ and hence also bounded in $W^{2,2,1}(B_R)$. Since $u_n\wto u$ in $W^{2,2}(B_1)$, we only need to show that both $\nabla u_n \to \nabla u$ and $\nabla^2 u_n\to \nabla^2 u$ strongly in $L^2(B_R)$. The first strong convergence is clear and we are left to show the second strong convergence.

Applying \cite[Lemma A.6]{Sharp-Topping-2013-TAMS} with $V_n=\nabla^2 u_n$, it suffices to show that
%\begin{equation}\label{eq:weak compactness nabla u}
%	\lim_{r\to 0}\limsup_{n\to \infty}\|\nabla u_n\|_{L^2(B_r(x))}=0
%\end{equation} 	
%and
\begin{equation}\label{eq:weak compactness Delta u}
		\lim_{r\to 0}\limsup_{n\to \infty}\|\nabla^2 u_n\|_{L^2(B_r(x))}=0.
\end{equation}

Apply (\ref{eq: decay estimate in LlogL}) with $\tau<1$ and scaling
we obtain
\[
\|\De u\|_{L^{2}(B_{\tau r})}+\|\nabla u\|_{L^{4,2}\left(B_{\tau r}\right)}\le\frac{1}{2}\left(\|\De u\|_{L^{2}(B_{r})}+\|\nabla u\|_{L^{4,2}\left(B_{r}\right)}\right)+C\|f\|_{L^{1}(B_{r})}^{1/2}\|f\|_{L\log L(B_{r})}^{1/2}.
\]
Applying Lemma \ref{lemma:ST 2.1} to the last term yields
\[
\|\De u\|_{L^{2}(B_{\tau r})}+\|\nabla u\|_{L^{4,2}\left(B_{\tau r}\right)}\le\frac{1}{2}\left(\|\De u\|_{L^{2}(B_{r})}+\|\nabla u\|_{L^{4,2}\left(B_{r}\right)}\right)+C\left(\log \frac 1r\right)^{-\frac 12}\|f\|_{L\log L(B_{r})}.
\]
Hence,
\[\lim_{r\to 0}\lim_{\tau\to 0}\limsup_{n\to \infty}\|\Delta u_n\|_{L^{2}(B_{\tau r})}^2\leq \frac{1}{2}\lim_{r\to 0}\lim_{\tau\to 0}\limsup_{n\to \infty}\|\Delta u_n\|_{L^{2}(B_{r})}^2,\]
from which we conclude
\[\lim_{r\to 0}\limsup_{n\to \infty}\|\Delta u_n\|_{L^{2}(B_{r})}=0.\]
This together with the standard $L^2$ theory for elliptic equations gives \eqref{eq:weak compactness Delta u}.
\end{proof}

\begin{remark}
In view of Remark \ref{rem: all decay in LlogL}, the above arguments also imply that $u_n\to u$  strongly in $W^{2,2,s}_{\loc}$ for all $1<s\le \wq$. But we can not conclude a strong convergence  in $W^{2,2,1}_{\loc}$. This seems to be a case on the borderline. Indeed, slightly strengthen the assumption by assuming $f\in L\log^pL(B_{10})$ for some $p>1$, then for any $0<r<1$, there holds
\[
\begin{aligned}\int_{B_{r}}|f|\log(2+|f|) & \le\left(\int_{B_{r}}|f|\log^{p}(2+|f|)\right)^{1/p}\left(\int_{B_{r}}|f|\right)^{1-\frac{1}{p}}\\
 & \lesssim\left(\log\frac{1}{r}\right)^{-\frac{p-1}{p}}\int_{B_{r}}|f|\log^{p}(2+|f|).
\end{aligned}
\]
Combining this inequality together with the decay estimate for $s=1$ in  Remark \ref{rem: all decay in LlogL}, the same arguments yield the strong convergence  in $W^{2,2,1}_{\loc}$.
\end{remark}

\appendix

\section{A note on Calder\'on-Zygmund estimate}\label{sec:C-Z theorey revisited}
The aim of this section is to prove that the Calder\'on-Zygmund estimate  is  locally uniform with respect to $p$.
\begin{proposition}\label{prop:local C_Z bound}
For each $\delta\in (0,\frac{1}{2})$, there exists $C=C_{\delta,n}$ such that for any $p\in [1+\delta,\frac{1+\delta}{\delta}]$, the following Calder\'on-Zygmund estimate holds	\begin{equation}\label{eq:local CZ-estimate}
		\|\na^{2}u\|_{p,B_{\frac{1}{2}}}\le C_{\de,n}\left(\|\De u\|_{p,B_{1}}+\|u\|_{p,B_{1}}\right)
\end{equation}
for all $u\in W^{2,p}(B_1)$.
\end{proposition}

%Define
%\[
%C_{p}=\sup_{u\in W^{2,p}(\R^{n}),u\neq0}\frac{\|\na^{2}u\|_{p}}{\|\De u\|_{p}}.
%\]
%We show that $C_{p}$ is locally uniformly bounded in $(1,\wq)$.
\begin{proof}

We first prove a global version. That is, for each $\delta\in (0,\frac{1}{2})$, there exists $C=C_{\delta,n}$ such that for any $p\in (1+\delta,\frac{1+\delta}{\delta})$, 	
\begin{equation}\label{eq:global CZ-estimate}
\|\na^{2}u\|_{L^p(\R^n)}\le C_{\de,n}\left(\|\De u\|_{L^p(\R^n)}+\|u\|_{L^p(\R^n)}\right)
\end{equation}
for all $u\in W^{2,p}(\R^n)$.

By the well-known estimates for Calder\'on-Zygmund operators (see e.g.~\cite{Grafakos-2008-classical}), we have
\[
\|\na^{2}u\|_{L^{1,\wq}(\R^{n})}\le C_{1,n}\|\De u\|_{L^{1}(\R^n)}
\]
and
\[
\|\na^{2}u\|_{L^2(\R^n)}\le C_{2,n}\|\De u\|_{L^2(\R^n)}.
\]
By \cite[Corollary 9.10]{Gilbarg-Trudinger-Book-2001}, we can  take $C_{2,n}=1$. For $1<p<2$ and $u\in W^{2,p}(\R^{n})$, the Marcinkiewicz interpolation theorem (see e.g. \cite[Theorem 9.8]{Gilbarg-Trudinger-Book-2001} or \cite[Theorem 1.3.2]{Grafakos-2008-classical}) implies
\[
\|\na^{2}u\|_{L^p(\R^n)}\le C_{p,n}\|\De u\|_{L^p(\R^n)},
\]
where
\[
C_{p,n}=2\left(\frac{p}{(p-1)(2-p)}\right)^{1/p}\left(C_{1,n}\right)^{\ta}\left(C_{2,n}\right)^{1-\ta}
\]
and  $\theta=\frac{2}{p}-1$. Thus, for any given $\delta \in (0,\frac{1}{2})$ and $1<p\le1+\de$, we have $2-p>1/2$ and $\theta\in (\frac{1}{3},1)$. Consequently, we infer that
\[
C_{p,n}\le\frac{C(n)}{p-1}.
\]

Fix $\de\in(0,1/2)$. We apply the Riesz-Thorin interpolation theorem (see e.g. \cite[Theorem 1.3.4]{Grafakos-2008-classical} with $p_{0}=q_{0}=1+\de$, $p_{1}=q_{1}=2$), to obtain, for any
$1+\de\le p\le2$,
\[
\|\na^{2}u\|_{L^p(\R^n)}\le C_{1+\de,n}^{\ta}C_{2,n}^{1-\ta}\|\De u\|_{L^p(\R^n)}\leq C_{\de,n}\|\De u\|_{L^p(\R^n)},
\]
where in the last inequality we used the fact that $\theta=\frac{2(p-1-\delta)}{p(1-\delta)}\leq \frac{2}{1+\delta}$.

For $2\le p\le\frac{1}{1-\frac{1}{1+\de}}=\frac{1+\de}{\de}$, we  conclude by duality that
\[
\|\na^{2}u\|_{L^p(\R^n)}\le C_{\de,n}\|\De u\|_{L^p(\R^n)}
\]
holds for all $u\in W^{2,p}(\R^{n})$. This proves \eqref{eq:global CZ-estimate}.

Now we can prove the local Calder\'on-Zygmund estimate \eqref{eq:local CZ-estimate}.

For any given $u\in W^{2,p}(B_{1})$, we extend $u$ to $\R^n$ as zero outside $B_1$ and choose $\eta\in C_{0}^{\wq}(B_{1})$ such that $\eta\equiv 1$ on $B_{\frac{1}{2}}$, $0\leq \eta\leq 1$ on $\R^n$ and $\max\{\|\nabla \eta\|_{L^\infty(\R^n)}, \|\De \eta\|_{L^\infty(\R^n)} \}\leq C_n$. Then for any $p\in [1+\de,\frac{1+\de}{\de}]$, we apply the previous global estimate to find a constant $C_{\de,n}>0$ such that
\[
\|\na^{2}(\eta u)\|_{L^p(\R^n)}\le C_{\de,n}\|\De(\eta u)\|_{L^p(\R^n)}.
\]
As a consequence, we have
\begin{equation}\label{eq:local estimate}
\begin{aligned}
\|\na^{2}u\|_{p,B_{1/2}}&\le C_{\de,n}\left(\|\eta \De u\|_{L^p(\R^n)}+2\|\na\eta\cdot\na u\|_{L^p(\R^n)}+\|\De\eta u\|_{L^p(\R^n)}\right)\\
&\le 2C_{\de,n}C_n\left(\|\De u\|_{p,B_{1}}+\|\na u\|_{p,B_{1}}+\|u\|_{p,B_{1}}\right).
\end{aligned}    	
\end{equation}

On the other hand, by the interpolation inequality for Sobolev spaces (see e.g. \cite[Theorem 7.28]{Gilbarg-Trudinger-Book-2001}),
there exists $C=C(n)$ such that for any $\ep>0$,
\[
\|\na u\|_{p,B_{1}}\le\epsilon\|\De u\|_{p,B_{1}}+\frac{C(n)}{\ep}\|u\|_{p,B_{1}}.
\]
Substituting the above inequality with $\ep=1$ into \eqref{eq:local estimate}, we finally obtain
\[
\|\na^{2}u\|_{p,B_{1/2}}\le C_{\de,n}\left(\|\De u\|_{p,B_{1}}+\|u\|_{p,B_{1}}\right).
\]
The proof is complete. \end{proof}

\section{A slightly improved Calder\'on-Zygmund estimate}\label{sec:an interpolation inequality revisited}

In this section, we  prove the following proposition which states a slightly improved Calder\'on-Zygmund estimate. It seems very possible that this proposition was already established in some literature. As we did not find a precise reference at hand, we present a detailed proof here for the reader's convenience.

\begin{proposition}\label{prop:Interpolation inequality}
For each $\delta\in (0,\frac{1}{2})$, there exists $C=C_{\delta,n}$ such that for any $p\in [1+\delta,n-\delta]$, the following Calder\'on-Zygmund estimate holds	
\begin{equation}\label{eq:interpolation inequality}
\|\na^{2}u\|_{p,B_{1/2}}\le C_{\de,n}\left(\|\De u\|_{p,B_{1}}+\|u\|_{1,B_{1}}\right)
\end{equation}
for all $u\in W^{2,p}(B_1)$.
\end{proposition}
In our case, $n=4$ and $1<p<4/3$, so $1<p/(2-p)<\ga<2p/(2-p)<4$.
Thus, there exists a constant $C$ independent of $\ga\in(\frac{p}{2-p},\frac{2p}{2-p})$
such that
\begin{equation}\label{eq: B2}
\|\na^{2}u\|_{\ga,B_{1/2}}\le C\left(\|\De u\|_{\ga,B_{1}}+\|u\|_{1,B_{1}}\right).
\end{equation}

\begin{proof}
We first recall the following results from \cite[Chapter 5]{Adams-book}.
\begin{itemize}

\item \textbf{P1}. There exists $C_{n}>0$ depending only on $n$ such that there
exists an extension operator $E:W^{1,p}(B_{1})\to W_{0}^{1,p}(B_{2})$
for all $1\le p<\wq$ satisfying
\[
\|Eu\|_{W^{1,p}(\R^n)}\le C_{n}\|u\|_{W^{1,p}(B_1)}.
\]

\item \textbf{P2}. Let $1<p<n$ and $u\in W^{1,p}(B_{1})$ for $B_{1}\subset\R^{n}$.
Then
\[
\|u\|_{p}\le\|u\|_{1}^{\ta}\|u\|_{p^{\ast}}^{1-\ta},
\]
where $\frac{1}{p}=1+(1-\theta)\big(\frac{1}{p^\ast}-1\big)$ or equivalently $\theta=\frac{p}{np-n+p}$.
%$\ta(1-\frac{1}{p}+\frac{1}{n})=\frac{1}{n}$.

\end{itemize}
%{\color{blue}Provide references for the above results. Maybe Adams-Fournier.}

By the Sobolev embedding theorem  and property \textbf{P1}, we have
\[
\|u\|_{p^{\ast},B_{1}}\le\|Eu\|_{p^{\ast},B_{2}}\le C_{p}\|\na(Eu)\|_{p,B_{2}}\le C_{p}C_{n}(\|u\|_{p,B_{1}}+\|\na u\|_{p,B_{1}}).
\]
Here $C_{p}$ is the best Sobolev constant satisfying
\[
C_{p}\le\frac{n-1}{\sqrt{n}}\frac{p}{n-p}.
\]
Thus,
\[
\|u\|_{p, B_1}\le\|u\|_{1, B_1}^{\ta}\left(C_{p}C_{n}(\|u\|_{p,B_{1}}+\|\na u\|_{p,B_{1}})\right)^{1-\ta}.
\]
Next we apply  the following interpolation theorem (see e.g.~\cite[Theorem 5.2]{Adams-book}): there exists $C_{n}>0$ such that for
all $1\le p<\wq$ and $\ep>0$,
\[
\|\na u\|_{p,B_{1}}\le\epsilon\|\De u\|_{p,B_{1}}+\frac{C(n)}{\ep}\|u\|_{p,B_{1}},
\]
Taking $\ep=1$, we obtain
\[
\|u\|_{p, B_1}\le\|u\|_{1, B_1}^{\ta}\left(C_{p}C_{n}(\|u\|_{p,B_{1}}+\|\De u\|_{p,B_{1}})\right)^{1-\ta}.
\]
Since $a^{\ta}b^{1-\ta}\le\ep^{-\frac{1-\ta}{\ta}}a+\ep b$, we have
\[
\|u\|_{p, B_1}\le\ep^{-\frac{1-\ta}{\ta}}\|u\|_{1, B_1}+\ep C_{p}C_{n}(\|u\|_{p,B_{1}}+\|\De u\|_{p,B_{1}}).
\]
Take $\ep=1/(2C_{p}C_{n})$ yields
\[
\|u\|_{p}\le(2C_{p}C_{n})^{\frac{1-\ta}{\ta}}\|u\|_{1}+\frac{1}{2}\|\De u\|_{p,B_{1}}.
\]

Note that $1/n\le\ta\le1$ for all $1<p<n$, and $\ta\to1$ as $p\to1$,
$\ta\to1/n$ as $p\to n$. Thus, $0\le1-\ta\le\frac{1-\ta}{\ta}\le n(1-\ta)\le n$
and so
\[
(2C_{p}C_{n})^{\frac{1-\ta}{\ta}}\le2^{n}C_{n}^{n}\left(\frac{p}{n-p}\right)^{\frac{1-\ta}{\ta}}
\]
is locally uniformly bounded for $p\in[1,n)$.

Finally, combining the above estimate with Proposition \ref{prop:local C_Z bound}, we obtain
\[
\|\na^{2}u\|_{p,B_{1/2}}\le C_{\de,n}C_{p}\left(\|\De u\|_{p,B_{1}}+\|u\|_{1,B_{1}}\right)
\]
with a constant $C_{\de,n}C_{p}$ uniformly bounded by $C(\delta,n)$ for all $p\in[1+\de,n-\delta]$.
\end{proof}

\section{A regularity result}
The following lemma is a special case of Theorem 3.31 of \cite{Ambrosio-Book}. We sketch the proof for readers' convenience.
\begin{lemma}\label{lem: a W1p regularty lemma}  Let $\boldsymbol{f}\in L^{p}(B_{1})$ and $v\in W^{1,p}(B_{1})$
for some $1<p<2$. Then there exists a unique $h\in W^{1,p}(B_{1})$
solving equation
\begin{equation}
\begin{cases}
-\De h=\Sd\boldsymbol{f} & \text{in }B_{1},\\
h=v & \text{on }\pa B_{1}.
\end{cases}\label{eq: Dirichlet BVP}
\end{equation}
Moreover, there exists a constant $C=C(n,p)>0$, such that
\[
\|h\|_{W^{1,p}(B_{1})}\le C\left(\|\boldsymbol{f}\|_{L^{p}(B_{1})}+\|v\|_{W^{1,p}(B_{1})}\right).
\]
\end{lemma}

\begin{proof} First assume $\boldsymbol{f},v\in C^{2}(\bar{B}_{1})$
such that there exists a solution $h$ of equation (\ref{eq: Dirichlet BVP}).
Let $\bar{h}=h-v\in W_{0}^{1,p}(B_{1})$. Then
\[
-\De\bar{h}=\Sd(\boldsymbol{f}+\na v)
\]
 in $B_{1}$. Put $F=|\na\bar{h}|^{p-2}\na\bar{h}\in L^{p^{\prime}}(B_{1})$. Note $p^{\prime}>2$.
Hodge decomposition gives a function $\var\in W_{0}^{1,p^{\prime}}(B_{1})$
and a function $G\in L^{p^{\prime}}(B_{1})$ satisfying $\Sd G=0$
in $B_{1}$, and
\[
\|\na\var\|_{L^{p^{\prime}}(B_{1})}+\|G\|_{L^{p^{\prime}}(B_{1})}\le C(n,p)\|F\|_{L^{p^{\prime}}(B_{1})}=C(n,p)\|\na\bar{h}\|_{L^{p^{\prime}}(B_{1})}^{p-1}.
\]
Thus, using this Hodge decomposition and H\"older's inequality, we
obtain
\[
\begin{aligned}\|\na\bar{h}\|_{L^{p}(B_{1})}^{p} & =\int_{B_{1}}\na\bar{h}\cdot F\\
 & =\int_{B_{1}}\na\bar{h}\cdot\na\var
 =-\int_{B_{1}}\left(\boldsymbol{f}+\na v\right)\cdot\na\var\\
 & \le\left(\|\boldsymbol{f}\|_{L^{p}(B_{1})}+\|\na v\|_{L^{p}(B_{1})}\right)\|\na\var\|_{L^{p^{\prime}}(B_{1})}\\
 & \le C(n,p)\left(\|\boldsymbol{f}\|_{L^{p}(B_{1})}+\|\na v\|_{L^{p}(B_{1})}\right)\|\na\bar{h}\|_{L^{p^{\prime}}(B_{1})}^{p-1}.
\end{aligned}
\]
This implies $\|\na\bar{h}\|_{L^{p}(B_{1})}\le C(n,p)\left(\|\boldsymbol{f}\|_{L^{p}(B_{1})}+\|\na v\|_{L^{p}(B_{1})}\right)$,
which in turn gives
\begin{equation}
\|\na h\|_{L^{p}(B_{1})}\le C\left(\|\boldsymbol{f}\|_{L^{p}(B_{1})}+\|\na v\|_{L^{p}(B_{1})}\right)\label{eq: estimate of gradient}
\end{equation}
for some $C>0$ depending only on $n$ and $p$.

Then, using Poincar\'e's inequality, we obtain
\[
\|h\|_{L^{p}(B_{1})}\le\|h-v\|_{L^{p}(B_{1})}+\|v\|_{L^{p}(B_{1})}\le C_{n,p}\|\na h-\na v\|_{L^{p}(B_{1})}+\|v\|_{L^{p}(B_{1})}.
\]
Combing the estimate (\ref{eq: estimate of gradient}) yields
\[
\|h\|_{L^{p}(B_{1})}\le C\left(\|\boldsymbol{f}\|_{L^{p}(B_{1})}+\|v\|_{W^{1,p}(B_{1})}\right).
\]
This completes the proof in the case $\boldsymbol{f},v\in C^{2}(\bar{B}_{1})$.

The general case follows from a standard approximation argument. We omit the details. The uniqueness  is also easy. \end{proof}

\textbf{Acknowledgements.} We would like to thank Xiao Zhong for many useful discussions during the preparation of this work. We would also like to thank the anonymous referees for many useful comments and suggestions that greatly improves the work.

\end{document}